%% file: Spin7InstantonsCayleys.tex
\documentclass[leqno,11pt]{article}
\input{preamble/preamble}

\author{
  Thomas Walpuski
}
\title{
  \texorpdfstring{$\Spin(7)$}{Spin(7)}--instantons, Cayley submanifolds, and Fueter sections
}
\date{2016-05-09}

\begin{document}
\maketitle

\begin{abstract}
  We prove an existence theorem for $\Spin(7)$--instantons, which are highly concentrated near a Cayley submanifold;
  thus giving a partial converse to Tian's foundational compactness theorem~\cite{Tian2000}.
  As an application, we show how to construct $\Spin(7)$--instantons on $\Spin(7)$--manifolds with suitable local $K3$ Cayley fibrations.
  This recovers an example constructed by \citet{Lewis1998}.
\end{abstract}

\input{intro}

\paragraph{Acknowledgements}
This article is the outcome of work undertaken by the author for his PhD thesis at Imperial College London~\cite{Walpuski2013}, supported by European Research Council Grant 247331.
I am grateful to my supervisor Simon Donaldson for his encouragement.

\input{review}
\input{asd}
\input{fc}
\input{approx}
\input{la}
\input{pf}
\input{fib}
\input{index}

\printreferences

\end{document}


%% file: preamble/preamble.tex
\usepackage[utf8]{inputenc}
\usepackage[T1]{fontenc}
\usepackage{microtype}

\usepackage[letterpaper]{geometry}

\ifdefined\screenview
  \edef\mtht{\the\textheight}
  \edef\mtwd{\the\textwidth}
  \usepackage{xcolor}
  \definecolor{BackgroundColor}{RGB}{253, 246, 227}
  \pagecolor{BackgroundColor}
\fi

\usepackage{amsmath}
\usepackage{amsthm}
\usepackage{tikz-cd}
\usetikzlibrary{arrows} 
\tikzset{
  commutative diagrams/.cd, 
  arrow style=tikz, 
  diagrams={>=stealth}
}
\usetikzlibrary{matrix,decorations.pathreplacing,calc}

\usepackage{textcomp}

\usepackage[sb]{libertine}
\usepackage[varqu,varl]{zi4}
\usepackage[libertine,bigdelims,vvarbb]{newtxmath} 
\usepackage[supstfm=libertinesups,%
  supscaled=1.2,%
  raised=-.13em]{superiors}
\useosf 
\usepackage[scr=boondox,cal=euler]{mathalfa}

\usepackage{slashed}
\usepackage{esint} 
\usepackage[english]{babel}

\usepackage{imakeidx}
\makeindex[intoc]

\usepackage{csquotes}
\usepackage[
  backend=biber,
  hyperref=true,
  backref=true,
  isbn=false,
  doi=true,
  natbib=true,
  eprint=true,
  useprefix=true,
  maxcitenames=99,
  maxbibnames=99,  
  maxalphanames=99, 
  minalphanames=99,
  safeinputenc,
  style=alphabetic,
  citestyle=alphabetic,
  block=space,
  datamodel=preamble/ext-eprint
]{biblatex}
\usepackage[
  hypertexnames = false,
  colorlinks    = true,
  citecolor     = gray,
  linkcolor     = gray,
  urlcolor      = gray,
  breaklinks
]{hyperref}
\makeatletter
\DeclareFieldFormat{arxiv}{%
  arXiv\addcolon\space
  \ifhyperref
    {\href{http://arxiv.org/\abx@arxivpath/#1}{%
       \nolinkurl{#1}%
       \iffieldundef{arxivclass}
         {}
         {\addspace\texttt{\mkbibbrackets{\thefield{arxivclass}}}}}}
    {\nolinkurl{#1}
     \iffieldundef{arxivclass}
       {}
       {\addspace\texttt{\mkbibbrackets{\thefield{arxivclass}}}}}}
\makeatother
\DeclareFieldFormat{mr}{%
  MR\addcolon\space
  \ifhyperref
    {\href{http://www.ams.org/mathscinet-getitem?mr=MR#1}{\nolinkurl{#1}}}
    {\nolinkurl{#1}}}
\DeclareFieldFormat{zbl}{%
  Zbl\addcolon\space
  \ifhyperref
    {\href{http://zbmath.org/?q=an:#1}{\nolinkurl{#1}}}
    {\nolinkurl{#1}}}
\renewbibmacro*{eprint}{%
  \printfield{arxiv}%
  \newunit\newblock
  \printfield{mr}%
  \newunit\newblock
  \printfield{zbl}%
  \newunit\newblock
  \iffieldundef{eprinttype}
    {\printfield{eprint}}
    {\printfield[eprint:\strfield{eprinttype}]{eprint}}
}
\AtEveryBibitem{%
  \clearlist{publisher}%
}
\AtEveryBibitem{%
  \clearlist{address}%
}
\DeclareFieldFormat[article,inproceedings,inbook,incollection,thesis]{title}{\textit{#1}}
\renewbibmacro{in:}{}
\newcommand{\printreferences}{\printbibliography[heading=bibintoc]}

\usepackage[inline,shortlabels]{enumitem}

\usepackage{subcaption}

\usepackage[yyyymmdd]{datetime}

\usepackage{etoolbox}
\ifundef{\abstract}{}{\patchcmd{\abstract}%
    {\quotation}{\quotation\noindent\ignorespaces}{}{}}

\usepackage[super]{nth}

\input{preamble/envs.tex}

\input{preamble/macros.tex}
\input{preamble/letters.tex}


%% file: preamble/envs.tex
\usepackage{aliascnt}

\numberwithin{equation}{section}

\renewcommand{\eqref}[1]{\hyperref[#1]{\rm(\ref*{#1})}}

\def\makeautorefname#1#2{\AtBeginDocument{\expandafter\def\csname#1autorefname\endcsname{#2}}}

\newcommand{\mynewtheorem}[2]{
  \newaliascnt{#1}{equation}          
  \newtheorem{#1}[#1]{#2}
  \aliascntresetthe{#1}
  \makeautorefname{#1}{#2}
}

\mynewtheorem{axiom}{Axiom}
\newtheorem*{axiom*}{Axiom}
\mynewtheorem{theorem}{Theorem}
\newtheorem*{theorem*}{Theorem}
\mynewtheorem{prop}{Proposition}
\newtheorem*{prop*}{Proposition}

\mynewtheorem{cor}{Corollary}
\mynewtheorem{construction}{Construction}
\mynewtheorem{lemma}{Lemma}
\mynewtheorem{conjecture}{Conjecture}
\newtheorem*{conjecture*}{Conjecture}
\mynewtheorem{hyp}{Hypothesis}
\newtheorem{step}{Step}

\numberwithin{substep}{step}
\makeautorefname{step}{Step}
\makeautorefname{substep}{Step}
\newtheorem{case}{Case}

\numberwithin{subcase}{case}
\makeautorefname{case}{Case}
\makeautorefname{subcase}{case}

\usepackage{etoolbox}
\AtBeginEnvironment{proof}{\setcounter{step}{0}}

\theoremstyle{remark}
\mynewtheorem{remark}{Remark}
\newtheorem*{remark*}{Remark}
\mynewtheorem{convention}{Convention}
\newtheorem*{convention*}{Convention}
\newtheorem*{conventions*}{Conventions}

\theoremstyle{remark}
\mynewtheorem{warning}{Warning}

\theoremstyle{definition}
\mynewtheorem{definition}{Definition}
\newtheorem*{definition*}{Definition}
\mynewtheorem{notation}{Notation}
\mynewtheorem{data}{Data}
\mynewtheorem{example}{Example}
\newtheorem*{example*}{Example}
\mynewtheorem{exercise}{Exercise}
\mynewtheorem{solution}{Solution}
\mynewtheorem{question}{Question}
\mynewtheorem{problem}{Problem}
\newtheorem*{question*}{Question}

\mynewtheorem{summary}{Summary}

\makeautorefname{table}{Table}        
\makeautorefname{chapter}{Chapter}
\makeautorefname{section}{Section}
\makeautorefname{subsection}{Section}
\makeautorefname{subsubsection}{Section}
\makeautorefname{footnote}{Footnote}
\AtBeginDocument{\def\itemautorefname~#1\null{(#1)\null}}
\AtBeginDocument{\def\equationautorefname~#1\null{(#1)\null}}

%% file: preamble/macros.tex
\let\C\undefined

\usepackage{bm}
\usepackage{mathtools} 
\usepackage{stmaryrd} 

\DeclareFontFamily{U}{mathx}{\hyphenchar\font45}
\DeclareFontShape{U}{mathx}{m}{n}{
      <5> <6> <7> <8> <9> <10>
      <10.95> <12> <14.4> <17.28> <20.74> <24.88>
      mathx10
      }{}
\DeclareSymbolFont{mathx}{U}{mathx}{m}{n}
\DeclareFontSubstitution{U}{mathx}{m}{n}
\DeclareMathAccent{\widecheck}{0}{mathx}{"71}
\DeclareMathAccent{\wideparen}{0}{mathx}{"75}

\DeclareMathOperator{\End}{End}

\DeclareMathOperator{\Fr}{Fr}
\DeclareMathOperator{\GL}{GL}

\DeclareMathOperator{\HF}{\HF}

\DeclareMathOperator{\Hol}{Hol}
\DeclareMathOperator{\Hom}{Hom}

\DeclareMathOperator{\PD}{PD}

\DeclareMathOperator{\coker}{coker}

\DeclareMathOperator{\im}{im}
\DeclareMathOperator{\ind}{index}

\DeclareMathOperator{\supp}{supp}

\DeclarePairedDelimiter{\Abs}{\|}{\|}

\DeclarePairedDelimiter{\set}{\lbrace}{\rbrace}
\def\({\left(}
\def\){\right)}
\def\<{\left\langle}
\def\>{\right\rangle}

\newcommand{\C}{{\mathbf{C}}}

\newcommand{\N}{{\mathbf{N}}}

\newcommand{\R}{\mathbf{R}}

\newcommand{\SO}{\mathrm{SO}}
\newcommand{\SU}{\mathrm{SU}}

\newcommand{\Spin}{\mathrm{Spin}}

\newcommand{\Z}{\mathbf{Z}}

\newcommand{\andq}{\text{and}\quad}

\newcommand{\ch}{\mathrm{ch}}

\newcommand{\co}{\mskip0.5mu\colon\thinspace}

\newcommand{\defined}[2][\key]{\def\key{#2}\textbf{#2}\index{#1}}

\newcommand{\del}{\partial}
\newcommand{\ev}{\mathrm{ev}}

\newcommand{\id}{\mathrm{id}}

\newcommand{\iso}{\cong}

\newcommand{\qandq}{\quad\text{and}\quad}

\newcommand{\qand}{\quad\text{and}}

\newcommand{\so}{\mathfrak{so}}
\newcommand{\spin}{\mathfrak{spin}}

\newcommand{\vol}{\mathrm{vol}}

\renewcommand{\H}{\mathbf{H}}

\renewcommand{\Re}{\operatorname{Re}}

\renewcommand{\epsilon}{\varepsilon}
\renewcommand{\setminus}{{\backslash}}

\renewcommand{\leq}{\leqslant}
\renewcommand{\geq}{\geqslant}

\makeatletter
\renewcommand*\env@matrix[1][*\c@MaxMatrixCols c]{%
  \hskip -\arraycolsep
  \let\@ifnextchar\new@ifnextchar
  \array{#1}}

\renewcommand\xleftrightarrow[2][]{%
  \ext@arrow 9999{\longleftrightarrowfill@}{#1}{#2}}
\newcommand\longleftrightarrowfill@{%
  \arrowfill@\leftarrow\relbar\rightarrow}
\makeatother


%% file: preamble/letters.tex
\newcommand{\rd}{{\rm d}}

\newcommand{\rG}{{\rm G}}

\newcommand{\rI}{{\rm I}}
\newcommand{\rII}{{\rm II}}
\newcommand{\rIII}{{\rm III}}

\newcommand{\bA}{{\mathbf{A}}}

\newcommand{\bE}{{\mathbf{E}}}

\newcommand{\bL}{{\mathbf{L}}}

\newcommand{\cH}{\mathcal{H}}

\newcommand{\sA}{\mathscr{A}}

\newcommand{\sG}{\mathscr{G}}

\newcommand{\sM}{\mathscr{M}}

\newcommand{\sO}{\mathscr{O}}

\newcommand{\fg}{{\mathfrak g}}

\newcommand{\fp}{{\mathfrak p}}
\newcommand{\fq}{{\mathfrak q}}

\newcommand{\fs}{{\mathfrak s}}

\newcommand{\fu}{{\mathfrak u}}

\newcommand{\fA}{{\mathfrak A}}

\newcommand{\fF}{{\mathfrak F}}

\newcommand{\fI}{{\mathfrak I}}

\newcommand{\fK}{{\mathfrak K}}
\newcommand{\fL}{{\mathfrak L}}
\newcommand{\fM}{{\mathfrak M}}

\newcommand{\fT}{{\mathfrak T}}

\newcommand{\fX}{{\mathfrak X}}
\newcommand{\fY}{{\mathfrak Y}}

\newcommand{\slD}{\slashed D}
\newcommand{\slS}{\slashed S}

%% file: intro.tex
\section{Introduction}

In this article we study some aspects of gauge theory on \defined{$\Spin(7)$--manifolds}, i.e., compact Riemannian $8$--manifolds with holonomy contained in the exceptional Lie group $\Spin(7)\subset \SO(8)$.
Every $\Spin(7)$--manifold $X$ comes equipped with a $4$--form $\Phi$, which is a calibration in the sense of \citet{Harvey1982}.
Submanifolds $Q \subset X$ which are calibrated by $\Phi$ are called \defined{Cayley submanifolds}.
The linear operator $*(\cdot\wedge\Phi)\co \Lambda^2\to \Lambda^2$ has eigenvalues $-1$ and $3$ and with eigenspaces of dimension $21$ and $7$ respectively;
and, in analogy with gauge theory on $4$--manifolds, we consider connections $A$ whose curvature satisfies the ``anti-self-duality'' condition
\begin{equation}
  \label{Eq_Spin7Instanton0}
  *(F_A\wedge \Phi) = -F_A.
\end{equation}
After gauge fixing, \eqref{Eq_Spin7Instanton0} becomes elliptic.
Solutions to \eqref{Eq_Spin7Instanton0}, commonly called \defined{$\Spin(7)$--instantons}, are absolute minimisers of the Yang--Mills functional.
These equations play an important rôle in the Donaldson--Thomas programme~\cite{Donaldson1998}  to develop gauge theory in higher dimensions and, by dimensional reduction, give rise to a plethora of interesting gauge theoretical equations in dimensions less than eight.

\citet{Tian2000} discovered that there is an interesting relation between gauge theory in higher dimension and calibrated geometry.
In particular, his foundational compactness result---extending work of \citet{Price1983,Uhlenbeck1982a,Nakajima1988}---predicts that a sequence $(A_i)$ of $\Spin(7)$--instantons \emph{could} degenerate by ``bubbling off ASD instantons transversely to a Cayley submanifold $Q$''.
More precisely, outside $Q$ the sequence $(A_i)$ converges smoothly (possibly after passing to a subsequence and changing gauge) and for each ~$x\in Q$ there exists a non-trivial ASD instanton $\fI(x)$ on $N_xQ \coloneq T_xQ^\perp$ whose pullback to $T_xX$ is the limit of a blowing up of the sequence $(A_i)$ around the point $x$.
The main result of this article gives sufficient conditions under which this phenomenon \emph{will} appear.

\begin{theorem}\label{Thm_A}
  Let $(X,\Phi)$ be a compact $\Spin(7)$--manifold.
  Suppose we are given:
  \begin{itemize}
  \item
    an (irreducible and) unobstructed $\Spin(7)$--instanton $A_0$ on a $G$--bundle $E_0$ over $X$,
  \item
    an unobstructed Cayley submanifold $Q$ and
  \item
    an unobstructed Fueter section $\fI$ of an instanton moduli bundle $\fM \to Q$ associated with $Q$ and $E_0|_Q$.
  \end{itemize}
  Then there exists a constant $\Lambda>0$ and a $G$--bundle $E$ together with a family of (irreducible and) unobstructed $\Spin(7)$--instantons $(A_\lambda)_{\lambda\in(0,\Lambda]}$ on $E$.
  Moreover, as $\lambda$ tends to zero $A_\lambda$ converges to $A_0$ on the complement of $Q$ and at each point $x\in Q$ an ASD instanton in the equivalence class given by $\fI(x)$ bubbles off transversely.
\end{theorem}

\begin{remark}
  We define the concepts of instanton moduli bundles and Fueter sections thereof in \autoref{Sec_Fueter}.
  For now, it shall suffice to say that $\fM$ is a bundle of moduli spaces and a Fueter section of $\fM$ is a section which satisfies a non-linear p.d.e.~similar to a Dirac equation.

  Unobstructedness is best understood as a notion of being in general position;
  see \autoref{Def_Spin7InstantonUnobstructed}, \autoref{Def_CayleyUnobstructed} and \autoref{Def_FueterUnobstructed}.
\end{remark}

The proof of \autoref{Thm_A} is based on combining a gluing construction with adiabatic limit techniques.
The analysis involved is similar to unpublished work by Brendle on the Yang--Mills equation in higher dimension~\cites{Brendle2003,Brendle2003a} and Pacard--Ritor\'e's work on the Allen--Cahn equation~\cite{Pacard2003}.
The basic ideas, which are discussed briefly at the beginning of \autoref{Sec_Approx} and \autoref{Sec_LinearAnalysis}, are quite simple;
however, the reader should be warned that some of the precise technical details are quite delicate.

\autoref{Thm_A} can be used as a tool to construct examples of $\Spin(7)$--instantons.
A particularly interesting situation, where our result can be applied, is if $X$ has a suitable local $K3$ Cayley fibration.

\begin{theorem}
  \label{Thm_B}
  Let $X$ be a compact $\Spin(7)$--manifold with holonomy equal to $\Spin(7)$.
  Suppose that $Q$ is a Cayley submanifold in $X$ which has self-intersection number zero, is diffeomorphic to a $K3$ surface whose induced metric is sufficiently close to a hyperkähler metric and suppose that the induced connection on $NQ$ is almost flat.
  Then there exists a $5$--dimensional family of $\Spin(7)$--instantons on a $\SU(2)$--bundle $E$ over $X$ with $c_2(E) = \PD[Q]$.

  Moreover, if $Q_1,\ldots, Q_k$ is a collection of $k$ disjoint Cayley submanifolds as above, then there exists a $(8k-3)$--dimensional family of $\Spin(7)$--instantons on a $\SU(2)$--bundle $E$ over $X$ with $c_2(E) = \sum_{i=1}^k \PD[Q_i]$.
\end{theorem}

Here is a concrete example.

\begin{example}
  \citet[Example 14.3.3]{Joyce2000} gives an example of a $\Spin(7)$--manifold which contains two disjoint Cayley submanifolds $Q_1$ and $Q_2$ of the kind required by above.
  Applying \autoref{Thm_B} in this situation recovers the example of a $\Spin(7)$--instanton described in Lewis' DPhil thesis~\cite{Lewis1998}.
  In fact, it produces examples with $c_2(E) = n\PD[Q_1] + m\PD[Q_2]$ for arbitrary $n,m \in \N$ by taking the $Q_3,\ldots$ to be slight perturbations of $Q_1$ and $Q_2$ (which exist because $X$ is locally fibred near $Q_1$ and $Q_2$).
\end{example}

Every Cayley submanifold as above gives rise to a local fibration of $X$ by Cayley submanifolds, see \autoref{Prop_CayleyFibration};
hence, we can use \autoref{Thm_B} to produce large families of $\Spin(7)$--instantons.
This can be compared with the situation on negative definite four-manifolds~\cite{Taubes1982}, in which one can construct ASD instantons concentrated around any finite number of points.

Let us end the introduction on a speculative remark.
Suppose that $X$ is a compact $\Spin(7)$--manifold together with a fibration $\pi\co X\to B$ to a compact base whose generic fibre is a $K3$ Cayley submanifold.
In view of the above one could hope (very optimistically) that one can show that the moduli space $\sM$ of $\Spin(7)$--instantons on the $\SU(2)$--bundle $E$ obtained by applying \autoref{Thm_B} to a generic fibre of $\pi$ is smooth (or only mildly singular), $5$--dimensional and can be compactified by adding $B$ to the boundary.
Then we can use $\sM$ to construct a cobordism between $B$ and the link of the singular set of $\sM$ much as in the original proof of Donaldson's theorem \cite{Donaldson1983}.
In particular, if $\sM \cup B$ is smooth and compact, then $B$ is null-cobordant and, hence, $\sigma(B)=0$.
Although there are currently no known examples of $\Spin(7)$--manifolds with (singular) $K3$ Cayley fibrations, the above might serve as an indication of what could be achieved using gauge theory on $\Spin(7)$--manifolds.


%% file: review.tex
\section{Review of \texorpdfstring{$\Spin(7)$}{Spin(7)}--geometry}
\label{Sec_Review}

We begin with a crash course in $\Spin(7)$--geometry, touching upon the basic concepts and facts relevant for this article.
For a more thorough and comprehensive discussion we refer the reader to Joyce's book \cite{Joyce2000}, specifically Chapter 10.

\subsection{\texorpdfstring{$\Spin(7)$}{Spin(7)}--manifolds}
\label{Sec_Spin7Manifolds}

In this section we approach $\Spin(7)$--geometry by thinking of the $4$--form $\Phi$, and not the metric, as the defining structure.
However, both points of view are essentially equivalent.

\begin{definition}
  A $4$--form $\Phi$ on an $8$--dimensional vector space $W$ is called \defined{admissible} if there exists a basis of $W$ in which it is identified with the $4$--form $\Phi_0$ on $\R^8$ defined by
  \begin{equation}\label{Eq_Phi0}
    \begin{split}
      \Phi_0
      &\coloneq e^{0123} - e^{0145} - e^{0167} - e^{0246} + e^{0257} - e^{0347} \\
      &\quad+ e^{4567} - e^{2367} - e^{2345} - e^{1357} + e^{1346} - e^{1256}.
    \end{split}
  \end{equation}
  Here we denote the standard basis of $(\R^8)^*$ by $(e^0,\ldots,e^7)$.
  The space of admissible forms on $W$ is denoted by $\sA(W)$.
\end{definition}

\begin{remark}
  An intrinsic characterisation of admissible forms can be found in \cite[Theorem 7.4 and Definition 7.5]{Salamon2010}.
\end{remark}

We use the following slightly unconventional definition, see~\autoref{Rmk_ConventionalSpin7} for the relation with the usual definition.

\begin{definition}
  $\Spin(7)$ is the subgroup of $\GL(\R^8)$ preserving the $4$--form $\Phi_0$ defined in \eqref{Eq_Phi0}.
\end{definition}

\begin{definition}
  A \defined{$\Spin(7)$--structure} on an $8$--dimensional manifold $X$ is an admissible $4$--form $\Phi\in\Gamma(\sA(TX))\subset \Omega^4(X)$.
  An $8$--manifold together with a $\Spin(7)$--structure is called an \defined{almost $\Spin(7)$--manifold}.
\end{definition}

\begin{prop}[{\cite[Theorem 9.1 and Theorem 7.4]{Salamon2010}}]
  $\Spin(7)$ is a simple, compact, connected and simply connected Lie group of dimension $21$.
  $\Spin(7)$ is a subgroup of $\SO(8)$.
\end{prop}

It follows that each almost $\Spin(7)$--manifold is canonically equipped with a metric $g_\Phi$ and an orientation.

\begin{definition}
  Let $(X,\Phi)$ be an almost $\Spin(7)$--manifold.
  The \defined{torsion} of the $\Spin(7)$--structure $\Phi$ is defined to be
  \begin{equation*}
    \nabla_{g_\Phi} \Phi.
  \end{equation*}
  If $\nabla_{g_\Phi} \Phi=0$, then $\Phi$ is called \defined{torsion-free} and $(X,\Phi)$ is called a \defined{$\Spin(7)$--manifold}.
\end{definition}

Compact $\Spin(7)$--manifolds with $\Hol(g_\Phi) = \Spin(7)$ are difficult to come by.
Joyce has developed two construction techniques, which yield a good number of examples, see \cites{Joyce1996b,Joyce1999,Joyce2000}.

A very simple example of a $\Spin(7)$--manifold is $(\R^8, \Phi_0)$.
We will use this as a local model and it will be useful to realise it as a special case of the following examples.

\begin{example}
  \label{Ex_HyperkahlerPair}
  If $(S,\omega_1,\omega_2,\omega_3)$ and $(T,\mu_1,\mu_2,\mu_3)$ are a pair of hyperkähler surfaces, then $(S\times T,\Phi)$ with
  \begin{equation}
    \label{Eq_HyperkahlerPair}
    \Phi \coloneq \vol_S + \vol_T - \sum_{i=1}^3 \omega_i\wedge\mu_i 
  \end{equation}
  is a $\Spin(7)$--manifold.
\end{example}

\begin{example}
  \label{Ex_G2xLine}
  If $(Y,\phi)$ is a $\rG_2$--manifold, then $(\R\times Y,\Phi)$ with
  \begin{equation}
    \label{Eq_G2xLine}
    \Phi \coloneq \rd t\wedge \phi + \psi
  \end{equation}
  and $\psi \coloneq \Theta(\phi) = *_\phi\phi$ is a $\Spin(7)$--manifold.
\end{example}

Taking $S = T = \R^4$ with $(\omega_1,\omega_2,\omega_3)=(\mu_1,\mu_2,\mu_3)$ a positive orthonormal basis of $\Lambda^+ \coloneq \Lambda^+(\R^4)^*$ in \autoref{Ex_HyperkahlerPair} and $Y = \R^7$ with $\phi = e^{123} - e^{145} - e^{167} - e^{246} + e^{257} - e^{347}$ in \autoref{Ex_G2xLine} both recover $(\R^8,\Phi_0)$.

The following linear algebra fact can be seen as the $\Spin(7)$--analogue of $\Lambda^2 = \Lambda^+ \oplus \Lambda^-$, the splitting into (anti)-self-dual two-forms on $\R^4$.

\begin{prop}[{\cite[Theorem 9.5]{Salamon2010}}]
  \label{Prop_Lambda2Splitting}
  Let $\Phi$ be an admissible $4$--form on an $8$--dimensional vector space $W$.
  Then $\Lambda^2 W^*$ splits as follows
  \begin{equation*}
    \Lambda^2W^* = \Lambda^2_7\oplus\Lambda^2_{21}
  \end{equation*}
  with
  \begin{align*}
    \Lambda^2_7
      &\coloneq \left\{ \alpha : *(\alpha \wedge \Phi) = 3\alpha \right\} \qand \\
    \Lambda^2_{21}
      &\coloneq \left\{ \alpha : *(\alpha \wedge \Phi) = -\alpha \right\}
       \iso \spin(7).
  \end{align*}
\end{prop}

\begin{remark}
  \label{Rmk_ConventionalSpin7}
  The action of $\Spin(7)$ on $\Lambda^2_7$ gives rise to a double cover $\Spin(7)\to\SO(7)$;
  hence, the above definition of $\Spin(7)$ agrees with the usual definition as the universal cover of $\SO(7)$.
\end{remark}

\autoref{Prop_Lambda2Splitting} induces an analogous splitting of $\Lambda^2T^*X$ for every almost $\Spin(7)$--manifold.
By slight abuse of notation we will denote the corresponding summands by $\Lambda^2_d$ as well.
We denote the projection onto $\Lambda^2_d$ by
\begin{align*}
  \pi_d\co \Lambda^2 T^*X \to \Lambda^2_d.
\end{align*}

The following propositions are easy to check via straight-forward computation.

\begin{prop}
  \label{Prop_Lambda2SplittingHyperkahlerPair}
  If $\Phi$ is the admissible $4$--form on a product of two quaternionic
  lines $S$ and $T$ defined as in \eqref{Eq_HyperkahlerPair}, then $\Lambda^2_7$ splits as
  \begin{equation*}
    \Lambda^2_{7} = \Lambda^2_3 \oplus \Lambda^2_4,
  \end{equation*}
  where
  \begin{gather*}
    \Lambda^2_3 = \bigoplus_{i=1}^3 \<\omega_i-\mu_i\> \qand  \\
    \Lambda^2_4 = \left\{ \<L\cdot,\cdot\> \in
      S^* \otimes T^* : L\in\Hom(S,T) ~\text{satisfying}~ \sum_{i} J_i L
      I_i = -3L \right\}.
  \end{gather*}
  Here $I_i$ and $J_i$ denote the complex structures on $S$ and $T$
  corresponding to $\omega_i$ and $\mu_i$ respectively.
\end{prop}

\begin{prop}
  \label{Prop_Lambda2SplittingG2xLine}
  If $\Phi$ is the admissible $4$--form on a product of $\R$ with a $7$--dimensional vector space $V$ equipped with a non-degenerate $3$--form $\phi$ defined as in \eqref{Eq_G2xLine}, then $\Lambda^2_7$ can be written as
  \begin{equation*}
    \Lambda^2_7 = \{ \rd t\wedge v^* + i(v)\phi : v \in  V \}
  \end{equation*}
  and $\Lambda^2_{21}$ can be written as
  \begin{equation*}
    \Lambda^2_{21} = \{ \rd t\wedge *_V(\alpha \wedge \psi) - \alpha : \alpha \in \Lambda^2 V^* \}
  \end{equation*}
  where $\psi : =\Theta(\phi)$.
\end{prop}

\begin{prop}[{\cite[Proposition 10.5.6]{Joyce2000}}]
  If $\Phi$ is a $\Spin(7)$--structure on $X$, then $X$ is spin and
  has a canonical spin structure with
  \begin{equation*}
    \slS^+=\Lambda^0\oplus \Lambda^2_7 \qandq \slS^-=\Lambda^1_8.
  \end{equation*}
  Moreover, if $\Phi$ is torsion-free, then $X$ admits a non-trivial
  parallel spinor.
\end{prop}

\subsection{\texorpdfstring{$\Spin(7)$}{Spin(7)}--instantons}
\label{Sec_Spin7Instantons}

Throughout the remainder of this section we fix a $\Spin(7)$--manifold $(X, \Phi)$.
Also, let $G$ be a (compact semi-simple) Lie group and $E$ a $G$--bundle over a $\Spin(7)$--manifold.

\begin{definition}
  A connection $A\in \sA(E)$ on $E$ is called a \defined{$\Spin(7)$--instanton} if it satisfies
  \begin{equation*}
    *(F_A\wedge\Phi) = -F_A
  \end{equation*}
  or equivalently
  \begin{equation}
    \label{Eq_Spin7Instanton}
     \pi_7(F_A) = 0.
  \end{equation}
\end{definition}

This equation originated in the physics literature \cite{Corrigan1983} and was introduced to a wider mathematical audience by Donaldson--Thomas \cite[Section~3]{Donaldson1998}.
$\Spin(7)$--instantons were the topic of Lewis' DPhil thesis \cite{Lewis1998};
in particular, he proposed the construction of one non-trivial example on a $\SU(2)$--bundle over a $\Spin(7)$--manifold with full holonomy $\Spin(7)$, cf.~\autoref{Sec_CayleyFibration}.
Recently, a construction for $\Spin(7)$--instantons on $\Spin(7)$--manifolds
arising from \cite{Joyce1999} was given by \citet{Tanaka2012}.

For us the following ``trivial'' examples will play an important role.

\begin{example}
  In the situation of \autoref{Ex_HyperkahlerPair} if $I$ is an ASD instanton over $T$, then its pullback to $S \times T$ is a $\Spin(7)$--instanton.
\end{example}

\begin{example}
  \label{Ex_Spin7InstantonG2xLine}
  In the situation of \autoref{Ex_G2xLine} if $A$ is a $\rG_2$--instanton over $Y$, then its pullback to $\R \times Y$ is a $\Spin(7)$--instanton.
\end{example}

If $A\in\sA(E)$ is a connection on $E$, we define $L_A\co \Omega^1(X,\fg_E) \to \Omega^0(X,\fg_E) \oplus \Omega^2_7(X,\fg_E)$ by
\begin{equation}
  \label{Eq_Spin7InstantonLinearisation}
  L_A(a) \coloneq \(\rd_A^*a, \pi_7(\rd_Aa)\).
\end{equation}
This is the linearisation of \eqref{Eq_Spin7Instanton} supplemented with the Coulomb gauge condition;
it also agrees with the negative Dirac operator on $X$ twisted by $\fg_E$.

\begin{remark}
  \label{Rmk_Spin7InstantonG2xLine}
  In the situation of \autoref{Ex_Spin7InstantonG2xLine} denote the pullback of $A$ by $\bA$.
  Identifying $\Omega^1(X,\fg_E)$ with $\Omega^0(\R\times Y, \R\oplus p_2^*T^*Y)$ and $\Omega^0(X,\fg_E)\oplus\Omega^2_7(X,\fg_E)$ with $\Omega^0(\R\times Y, \R\oplus p_2^*T^*Y)$ using \autoref{Prop_Lambda2SplittingG2xLine}, we can write
  \begin{equation*}
    L_\bA = \del_t -
    \begin{pmatrix}
      0 & \rd_A^* \\
      \rd_A & *_Y(\psi\wedge\rd_A)
    \end{pmatrix}.
  \end{equation*}
  Note that the second term is nothing but the linearisation of the $\rG_2$--instanton equation at $A$, see~\cite[Section 3]{Walpuski2011}.
\end{remark}

\begin{prop}
  \label{Prop_Spin7InstantonIndex}
  If $A$ is a $\Spin(7)$--instanton, then there is an open subset $U\subset \ker L_A$ and a smooth map $\kappa\co U \to \coker L_A$ such that the moduli space of $\Spin(7)$--instantons near $A$ is homeomorphic to $\kappa^{-1}(0)/\Gamma_A$.
  Here $\Gamma_A\subset \sG(E)$ is the group of gauge transformations fixing $A$.
  The index of $L_A$ is given by
 \begin{equation}\label{eq:s7instanton_index}
   \begin{split}
     \ind L_A &= \dim \fg\cdot (b^1-b^0-b^2_7) \\
     &\quad    +\frac{1}{24}\int_X p_1(X)p_1(\fg_E)
     -\frac{1}{12}\int_X p_1(\fg_E)^2-2p_2(\fg_E).
   \end{split}
  \end{equation}
  If $E$ is a $\SU(r)$--bundle, then
  \begin{equation}\label{eq:s7instanton_index_sun}
    \begin{split}
    \ind L_A &=
    (r^2-1)(b^1-b^0-b^2_7) \\
    &\quad -
    \frac{r}{12} \int_X p_1(X)c_2(E) -
    \int_X \(1+\frac{r}{6}\)c_2(E)^2
    -\frac{r}{3}c_4(E).
    \end{split}
  \end{equation}
  Here $b^2_7$ is the refined second Betti number corresponding to $\Lambda^2_7$ in \autoref{Prop_Lambda2Splitting}, see \cite[Definition 10.6.3]{Joyce2000}.
\end{prop}

\begin{remark}
  The index formula given by Lewis \cite[Theorem 3.2]{Lewis1998} is incorrect.
  He mistakenly couples the Dirac operator to $E$ instead of $\fg_E$.
\end{remark}

\begin{proof}[Proof of the index formula]
  The existence of the Kuranishi map $\kappa$ is standard (see, e.g., \cite[Section 4.2]{Donaldson1990});
  we only prove the index formula.
  Using
  \begin{gather*}
    \ch_2(\fg_E\otimes \C)=-c_2(\fg_E\otimes \C) \qand \\
 \ch_4(\fg_E\otimes \C)=\frac{1}{12}(c_2(\fg_E\otimes \C)^2-2 c_4(\fg_E\otimes \C))
  \end{gather*}
  the index theorem yields
  \begin{align*}
    \ind L_A &= -\int_X \hat A(X) \ch(\fg_E\otimes \C)\\
    &=- \int_X
    \left(1-\frac{p_1(X)}{24}+\frac{7p_1(X)^2-4p_2(X)}{5670}\right)
    \\
    &\qquad\qquad\cdot
    \left(\dim \fg + p_1(\fg_E)
      +\frac{p_1(\fg_E)^2-2p_2(\fg_E)}{12} \right) \\
    &= \dim \fg\cdot (b^1-b^0-b^2_7) \\
    &\qquad
    +\frac{1}{24}\int_X p_1(X)p_1(\fg_E)
    -\frac{1}{12}\int_X p_1(\fg_E)^2-2p_2(\fg_E).
  \end{align*}
  In the last step, we applied the identity derived up to this point with $E$ the trivial line bundle to obtain
  \begin{equation*}
    b^0-b^1+b^2_7 
    = 
    \int_X \frac{7p_1(X)^2-4p_2(X)}{5670}.
  \end{equation*}
  If $E$ is a $\SU(r)$--bundle, then we can use
  \begin{equation*}
    \ch(\fg_E\otimes \C)=\ch(E\otimes E^*)-1
    =r^2-1 - 2rc_2(E) + \frac{6+r}{6} c_2(E)^2 - \frac{r}3 c_4(E).
    \qedhere
  \end{equation*}
\end{proof}

\begin{definition}
  \label{Def_Spin7InstantonUnobstructed}
  If $A$ is a $\Spin(7)$--instanton, then we denote by
  \begin{align*}
    \cH^0_A &\coloneq \ker L_A^* \cap \Omega^0(X,\fg_E),  \\
    \cH^1_A &\coloneq \ker L_A \cap \Omega^1(X,\fg_E) \qand \\
    \cH^2_{7;A} &\coloneq \ker L_A^* \cap \Omega^2_7(X,\fg_E)
  \end{align*}
  the \defined{space of infinitesimal automorphisms}, the \defined{space of infinitesimal deformations} and the \defined{space of infinitesimal obstructions} respectively.
  $A$ is called \defined{irreducible} if $\cH^0_A=0$ and \defined{unobstructed} if $ \cH^2_{7;A}=0$.
\end{definition}

\begin{remark}
  The above spaces can also be seen as the cohomology groups of the deformation complex
  \begin{equation*}
    0 \to \Omega^0(X,\fg_E) \xrightarrow{\rd_A} \Omega^1(X,\fg_E)  \xrightarrow{\pi_7\circ\rd_A} \Omega^2_7(X,\fg_E) \to 0.
  \end{equation*}
\end{remark}

\subsection{Cayley submanifolds}
\label{Sec_Cayley}

\begin{theorem}[{\citet[Chapter IV Theorem 1.24]{Harvey1982}}]
  \label{Thm_Cayley}
  If $(X,\Phi)$ is a $\Spin(7)$--manifold, then $\Phi$ is a calibration.
  Moreover, $Q\subset X$ is calibrated by $\Phi$ if and only of at each point $x\in Q$ there exists a basis $(e_0,\ldots,e_7)$ of $T_xX$ with respect to which $\Phi$ is given by \eqref{Eq_Phi0} and $(e_0,\ldots, e_3)$ is a positive basis of $T_xQ$.
\end{theorem}

\begin{remark}
  Recall that a differential $k$--form $\alpha$ on a Riemannian manifold $(M,g)$ is called a calibration if it is closed and has comass at most 1, that is, $\rd \alpha = 0$ and for all orthogonal subset $\{e_1,\ldots,e_k\} \subset T_xM$ we have $\alpha(e_1,\ldots,e_k) \leq 1$.
\end{remark}

\begin{definition}
  Let $(X,\Phi)$ be a $\Spin(7)$--manifold.
  Then $\Phi$ is called the \defined{Cayley calibration}.
  An oriented $4$--dimensional submanifold $Q\subset X$ that is calibrated by $\Phi$ is called a \defined{Cayley submanifold}.
\end{definition}

If $Q\subset (X,\Phi)$ is a Cayley submanifold, then it follows from \autoref{Thm_Cayley} that there is a natural identification
\begin{equation}
  \label{Eq_Lambda+}
  \Lambda^+ T^*Q \iso \Lambda^+ N^*Q.
\end{equation}
We define a subbundle $\Hom_\Phi(TQ,NQ)\subset \Hom(TQ,NQ)$ by decreeing that $L \in \Hom_\Phi(TQ,NQ)$ if and only if
\begin{equation*}
  \sum_{i} I_i L I_i =-3L,
\end{equation*}
cf.~\autoref{Prop_Lambda2SplittingHyperkahlerPair}.
Here $I_i$ runs through a local orthonormal basis of $\Lambda^+ T^*Q \iso \Lambda^+ N^*Q$, which we can identify with subsets of $\so(TQ)$ and $\so(NQ)$. 
Up to multiplication by $\frac14$
\begin{equation*}
  \gamma L \coloneq L - \sum_{i} I_i L I_i
\end{equation*}
defines a projection of $\Hom(TQ,NQ)$ onto $\Hom_\Phi(TQ,NQ)$.
\begin{definition}
  The \defined{Fueter operator} $F_Q\co \Gamma(Q,NQ)\to \Gamma(Q,\Hom_\Phi(TQ,NQ))$ associated with $Q$ is defined by
  \begin{equation*}
    F_Q(n) \coloneq \gamma(\bar\nabla n).
  \end{equation*}
\end{definition}

\begin{remark}
  If $e_0$ is a vector in $TQ$, then one can compose $F_Q$ with evaluation on $e_0$ to obtain the operator
  \begin{equation*}
    \ev_{e_0} \circ F_Q (n) = \bar\nabla_{e_0} n
    - \sum_i I_i \bar\nabla_{e_i} n
  \end{equation*}
  where $e_i\coloneq I_i e_0$.
  It is therefore appropriate to think of $F$ as a Dirac-type operator.
\end{remark}

\begin{remark}
  \label{Rmk_CayleySpin}
  Suppose that $Q$ is spin and $\fs$ is a spin structure on $Q$.
  Then the normal bundle $NQ$ is also spin, since $X$ is;
  moreover, there is a spin structure $\fu$ on $NQ$ such that $\slS^+_Q=\slS^+_{NQ}$ because of \eqref{Eq_Lambda+}.
  If we set $U\coloneq\slS^-_{NQ}$, then it can be seen that $\Re(\slS^+_Q\otimes U) = NQ$, $\Re(\slS^-_Q\otimes U) = \Hom_\Phi(TQ,NQ)$ and that $F_Q$ agrees with the twisted Dirac operator $\slD\co \Gamma(\Re(\slS^+_Q\otimes U)) \to \Gamma(\Re(\slS^-_Q\otimes U))$.
  For more details we refer the reader to \cite[Section 6]{McLean1998} and \cite[Section 3.2]{Haydys2011}.
\end{remark}

\begin{theorem}[{\citet[Section 6]{McLean1998}}]
  \label{Thm_McLean}
  Let $(X,\Phi)$ be a compact $\Spin(7)$--manifold and let $Q\subset X$ be a compact Cayley submanifold.
  Then there is an open subset $\sO\subset \ker F_Q$ and a smooth map $\kappa\co \sO \to \coker F_Q$ such that the moduli space of Cayley submanifolds near $Q$ is homeomorphic to $\kappa^{-1}(0)$.
  The index of $F_Q$ is given by
  \begin{equation}
    \label{Eq_CayleyIndex}
    \ind F_Q = \frac{\sigma(Q) + \chi(Q)}{2} - [Q]\cdot[Q].
  \end{equation}
  Here $\sigma(Q) : =b^+(Q)-b^-(Q)$ denotes the signature of $Q$.
\end{theorem}

\begin{remark}
  The index formula given by \citet[Equation (10.32)]{Joyce2000} is incorrect and likely a misprint as it also contradicts his remarks at the bottom of p.\,267.
\end{remark}

\begin{definition}
  \label{Def_CayleyUnobstructed}
  A Cayley submanifold $Q$ is called \defined{unobstructed} if $F_Q$ is surjective.
\end{definition}

\begin{proof}[Proof of the index formula]
  We can assume that $Q$ is spin.
  Then the index of $F_Q$ agrees with the index of the twisted Dirac operator $\slD_U$.
  By the Atiyah--Singer index theorem
  \begin{equation*}
    \ind \slD_U = \int_Q \hat A(Q) \ch_2(U) = -\frac{1}{4} \sigma(Q) -
    \int_Q c_2(U).
  \end{equation*}
  This is the formula given by McLean.
  In order to obtain a more useful expression, we make use of the fact that if $E$ and $F$ are a pair of $\SU(2)$--bundles over a $4$--manifold and $V=\Re(E\otimes F)$, then
  \begin{equation}
    \label{Eq_ep1c2}
    \begin{split}
      e(V) & =c_2(F)-c_2(E) \qand \\
      p_1(V) &= -2(c_2(E)+c_2(F)).
    \end{split}
  \end{equation}
  To see this, note that there must be universal formulas of the form $e(V)=\alpha(c_2(E)-c_2(F))$ and $p_1(V)=\beta(c_2(E)+c_2(F))$, because $e(V)$ changes sign when $E$ and $F$ are interchanged since this changes the orientation on $V$, and $p_1(V)$ is independent of the order of $E$ and $F$.
  The constants can be determined by a simple explicit computation for the spin bundles over $K3$.
  From these formulae it follows that
  \begin{equation*}
    c_2(U) = -\frac14\(p_1(NQ)-2e(NQ)\).
  \end{equation*}
  To compute $p_1(NQ)$, we combine $\slS^+_Q=\slS^+_{NQ}$ and \eqref{Eq_ep1c2} to obtain
  \begin{equation}
    \label{Eq_C2SpinorBundle}
    p_1(NQ)+2e(NQ) = -4c_2(\slS^+_{NQ}) = -4c_2(\slS^+_{Q}) = p_1(Q)+2e(Q);
  \end{equation}
  hence,
  \begin{equation}
    \label{Eq_p1nq}
    \int_Q p_1(NQ) = 3\sigma(Q)+2\chi(Q)-2[Q]\cdot[Q].
  \end{equation}
  Therefore,
  \begin{equation*}
    \int_Q c_2(U) = -\frac34\sigma(Q) - \frac12\chi(Q) + [Q]\cdot[Q],
  \end{equation*}
  which implies the claimed index formula.
\end{proof}

\begin{prop}
  \label{Prop_CayleyFibration}
  Let $X$ be a compact $\Spin(7)$--manifold.
  Suppose that $Q$ is a compact Cayley submanifold in $X$ which has self-intersection number zero, is diffeomorphic to a $K3$ surface whose induced metric is sufficiently close to a hyperkähler metric and suppose that the induced connection on $NQ$ is almost flat.
  Then $X$ is locally fibred by Cayley $K3$ surfaces near $Q$.
\end{prop}

\begin{proof}
  Using the fact that $Q$ and hence $NQ$ is spin as well as \eqref{Eq_ep1c2} one can show that $NQ$ is trivial.
  The Fueter operator $F_Q$ thus agrees with the Dirac operator $\slD_U\co \Gamma(\Re(\slS^+\otimes U))\to\Gamma(\Re(\slS^-\otimes U))$.
  On a hyperkähler $K3$ surface the untwisted Dirac operator $\slD$ is surjective, has a four-dimensional kernel, and every non-zero element of $\ker \slD$ is nowhere vanishing;
  hence, the same is true for $\slD_U$ because the  metric on $Q$ is
  sufficiently close to a hyperkähler metric and the connection on $U$ is almost flat.
  The existence of the local fibration now follows from (the proof of) \autoref{Thm_McLean}.
\end{proof}


%% file: asd.tex
\section{Moduli spaces of ASD instantons over \texorpdfstring{$\R^4$}{R4}} 
\label{Sec_ModuliOfASDInstantons}

This section is intended to remind the reader of some basic facts about ASD instantons over $\R^4$, all of which are completely classical and most of which can be found in Donaldson--Segal~\cite[Section 6.1]{Donaldson2009}.

Fix a $G$--bundle $E$ over $S^4=\R^4\cup\{\infty\}$.
Denote by $M$ the moduli space of ASD instantons on $E$ framed over the point at infinity, i.e.,
\begin{equation*}
  M(E) \coloneq \set{ A \in \sA(E) : F_A^+ = 0 }/\sG_0.
\end{equation*}
Here $\sA(E)$ denotes the space of connections on $E$ and
\begin{equation*}
  \sG_0(E) \coloneq \set{ g \in \sG(E) : g|_{E_\infty} = \id }
\end{equation*}
denotes the based gauge group.
These moduli spaces are smooth manifolds, because ASD instantons over $S^4$ are always unobstructed as a consequence of the Weitzenböck formula, see, e.g., \cite[Proposition~2.2]{Taubes1982}.
By Uhlenbeck's removable singularities theorem \cite[Theorem~4.1]{Uhlenbeck1982} we can think of $M$ as a \defined{moduli space of framed finite energy ASD instantons} on $\R^4$.
In a suitable functional analytic setup incorporating decay conditions at infinity, see, e.g., \cite{Taubes1983} or \cite{Nakajima1990}, the infinitesimal deformation theory of a framed ASD instanton $I$ over $\R^4$ is governed by the linear operator $\delta_I \co \Omega^1(\R^4,\fg_E) \to \Omega^0(\R^4,\fg_E) \oplus \Omega^+(\R^4,\fg_E)$ defined by
\begin{align}
  \label{Eq_DeltaI}
  \delta_I a \coloneq (\rd_I^*a,\rd_I^+a).
\end{align}
From the work of Taubes \cite{Taubes1983} it is known that $\delta_I$ is always surjective and that its kernel lies in $L^2$.
More precisely, we have the following result whose proof can be found, e.g., in \cite[Proposition~5.10]{Walpuski2011}.

\begin{prop}
  \label{Prop_ASDDecay}
  Let $E$ be a $G$--bundle over $\R^4$ and let $I\in\sA(E)$ be a finite energy ASD instanton on $E$.
  Then the following holds.
  \begin{enumerate}
  \item
    If $a\in\ker\delta_I$ decays to zero at infinity, that is to say $\lim_{r\to\infty} \sup_{\del B_r(0)} |a|=0$, then $|\nabla^k a|=O(r^{-3-k})$ for $k \geq 0$.
    Here $r\co\R^4\to[0,\infty)$ denotes the radius function $r(x)\coloneq |x|$.
  \item
    If $(\xi,\omega)\in\ker\delta_I^*$ decays to zero at infinity, then $(\xi,\omega)=0$.
  \end{enumerate}
\end{prop}

In particular, this implies $M$ can be equipped with the $L^2$--metric arising from the standard metric on $\R^4$.
Any self-dual $2$--form $\omega \in S(\Lambda^+)$ of unit length, determines a complex structure $J_\omega$ on $\R^4$ via $\Lambda^2(\R^4)^* \iso \so(4)$.
This makes $\R\oplus \Lambda^+$ into an algebra, which is abstractly isomorphic to the quaternions $\H$.
A key fact is that $\delta_I$ commutes with the action of this algebra \cite[Proof of Theorem 3.2]{Taubes1983};
hence, $T_{[I]}M = \ker \delta_I \subset \Omega^1(\R^4,\fg_E)$ is preserved.

\begin{prop}
  The $L^2$--metric and the complex structures $\{ J_\omega : \omega \in S(\Lambda^+) \}$ define a hyperkähler structure on $M$.
\end{prop}

This structure is $\SO(4)$--equivariant.
$M$ carries an action of $\R^4\rtimes \R^+$ where $\R^4$ acts by translation and $\R^+$ acts by dilation, i.e., by pullback via $s_\lambda$ where
\begin{align*}
  s_\lambda(x)\coloneq \lambda x
\end{align*}
for $\lambda\in\R^+$.
Since the centre of mass of the measure $|F_I|^2\vol$ is equivariant with respect to the $\R^4$--action, we can write
\begin{equation*}
  M = M^\circ \times \R^4
\end{equation*}
where $M^\circ$ is the space of instantons centred at zero.
The action of $\Lambda^+ \subset \Lambda^2 \iso \so(4)$ preserves this product structure and $\Lambda^+$ acts on the factor $\R^4$ in the usual way.

\begin{example}
  \label{Ex_ChargeOne}
  If $E$ is the unique $\SU(2)$--bundle over $S^4$ with $c_2(E)=1$, then $E$ carries a single ASD instanton $I$, commonly called ``the one-instanton'', unique up to scaling, translation and changing the framing at infinity.
  We can naturally write the corresponding moduli space as
  \begin{equation*}
    M = M^\circ\times\R^4
      = (\Re(\Hom(\C^2,\slS^+))\setminus\set{0})/\Z_2 \times \R^4.
  \end{equation*}
  Here $\slS^+$ is the positive spin representation associated with $\R^4$ and $\C^2$ has to be thought of as a $\SU(2)$ representation.
  In this situation both $\C^2$ and $\slS^+$ have canonical quaternionic structures and thus $\Hom(\C^2,\slS^+)$ inherits a real structure.
  The real part are simply the quaternionic-linear homomorphisms.
  The reader can consult \cite[Section 3.1]{Donaldson1990} for a more extensive discussion.
\end{example}

\begin{example}
  In general, if $E$ is an $\SU(2)$--bundle over $S^4$, then $M$ can be understood rather explicitly in terms of the ADHM construction \cite[Section~3.3]{Donaldson1990}.
\end{example}

\begin{prop}
  \label{Prop_UniversalConnection}
  There exists a $G$--bundle $\bE$ over $M \times S^4$ together with a framing $\bE|_{M\times\set\infty} \to G$ and a tautological connection $\bA \in \sA(\bE)$ on $\bE$ such that:
  \begin{itemize}
  \item
    $\bE|_{\set{[I]}\times S^4} \iso E$ and
  \item
    $\bA$ restricted to ${\set{[I]}\times \R^4}$ is equivalent to $[I]$ via $\sG_0(E)$.
  \end{itemize}
  If we decompose the curvature of the tautological connection $\bA$ over $M\times \R^4$ according to the bi-grading on $\Lambda^* T^*(M\times\R^4)$ induced by $T(M\times\R^4) = \pi_1^*TM \oplus \pi_2^*T\R^4$, then its components satisfy the following:
  \begin{itemize}
  \item
    $F_{\bA}^{2,0} = -2\Delta_I^{-1}\<[a,b]\>$.
  \item
    $F_{\bA}^{1,1} \in \Gamma(\Hom(\pi_1^*TM,\pi_2^*T\R^4\otimes \fg_\bE))$ at $([I],x)$ is the evaluation of $a \in T_{[I]}M = \ker \delta_I$ at $x$;
    in particular, it is $(\R\oplus\Lambda^+)$--linear.
  \item
    $F_{\bA}^{0,2} \in \Gamma(\pi_2^*\Lambda^-(\R^4)^* \otimes \fg_\bE)$.
  \end{itemize}
\end{prop}

\begin{proof}[Proof sketch]
  There is a tautological connection on the pullback of $E$ to $\sA(E) \times S^4$.
  It is flat in the $\sA(E)$--direction.
  It is $\sG_0$--equivariant, but not basic;
  hence, induces a connection on $M\times S^4$ after choosing a connection on $\sA(E) \to \sA(E)/\sG_0(E)$.
  We chose the connection given whose horizontal distribution is given by the Coulomb gauge with respect to the metric on $\R^4$;
  that is, the connection with connection $1$--form $\theta(a) = \Delta_I^{-1}\rd_I^* a$ for $a \in T_I\sA = \Omega^1(\R^n,\fg_E)$.
  The $(2,0)$--component of the curvature of $\bA$ arises from the curvature of this connection.
  The second two bullets are tautological.
\end{proof}


%% file: fc.tex
\section{Fueter sections of instanton moduli bundles over Cayley submanifolds}
\label{Sec_Fueter}
 
We now discuss models of $\Spin(7)$--instantons which are highly concentrated near a Cayley submanifold $Q$ in a $\Spin(7)$--manifold $(X,\Phi)$.

\subsection{The flat model}

We begin with studying the situation on $\R^8 = \R^4\times \R^4$.
Fix a basis $(\omega_1,\omega_2,\omega_3)$ of $\Lambda^+ \coloneq \Lambda^+ (\R^4)^*$ satisfying
\begin{equation*}
  \omega_i \wedge\omega_j = 2\delta_{ij} \vol
\end{equation*}
with $\vol$ denoting the standard volume form on $\R^4$.
Set $J_i \coloneq J_{\omega_i}$.
The standard $\Spin(7)$--structure $\Phi$ on $\R^8 = \R^4 \times \R^4$ can be written as
\begin{equation*}
  \Phi \coloneq \pi_1^*\vol + \pi_2^* \vol - \sum_{i=1}^3 \pi_1^*\omega_i \wedge \pi_2^*\omega_i.
\end{equation*}

It is a straight-forward computation, using \autoref{Prop_Lambda2SplittingHyperkahlerPair}, to check that:
\begin{prop}
  \label{Prop_InstantonComponentwise}
  A connection $A$ on a $G$--bundle $\pi_2^*E$ is a $\Spin(7)$--instanton if and only if:
  \begin{itemize}
  \item
    $\(F_A^{2,0}\)^+ = \(F_A^{0,2}\)^+$ and
  \item
    $F_A^{1,1}$ thought of as map $L \co T\R^4 \to \Omega^1(\R^4,\fg_E)$ satisfies
    \begin{equation}
      \label{Eq_Fueter}
      L - \sum\nolimits_{i=1}^3 J_i \circ L \circ J_i = 0.
    \end{equation}
  \end{itemize}
\end{prop}

Let $U$ be an open subset of $\R^4$.
Suppose $A_i$ is a sequence of $\Spin(7)$--instantons on $U \times \R^4$ on $\pi_2^*E$ concentrating along $U \times \set{0}$ and $(\lambda_i)$ is a null-sequence such that $[(x,y) \mapsto (x,\lambda_i y)]^*A_i$ converges to $A$.
Then it follows from \eqref{Prop_InstantonComponentwise} that
\begin{itemize}
\item
  $\(F_A^{0,2}\)^+ = 0$ and
\item
  $F_A^{1,1}$ satisfies \eqref{Eq_Fueter}.
\end{itemize}
By the first bullet, such an $A$ determines a map $\fI \co U \to M$ and by the second bullet this map satisfies the Fueter equation
\begin{equation*}
  \nabla \fI - \sum_{i=1}^3 J_i \circ \nabla\fI \circ J_i = 0.
\end{equation*}
Up to gauge equivalence, $A$ can be reconstructed from $\fI$ by pulling back the tautological connection on $M\times \R^4$ via $\fI \times \id_{\R^4}$.
Thus, Fueter maps into $M$ can serve as models for highly concentrated $\Spin(7)$--instantons on $U\times \R^4$.

\subsection{The model on \texorpdfstring{$NQ$}{NQ}}

We now globalise the above discussion.
Fix a moduli space $M$ of framed ASD instantons on a $G$--bundle $E$ over $\R^4$, as in \autoref{Sec_ModuliOfASDInstantons} and denote by $E_\infty$ a $G$--bundle over $Q$ together with a connection $A_\infty$.

\begin{definition}
  The \defined{instanton moduli bundle} $\fM\to Q$ associated with $Q$, $E_\infty$ and $M$ is defined by
  \begin{equation*}
    \fM \coloneq (\Fr(NQ)\times E_\infty)\times_{\SO(4)\times G} M.
  \end{equation*}
\end{definition}

\begin{example}\label{Ex_MChargeOne}
  If $M=(\Re(\Hom(\C^2,\slS^+))\setminus\{0\})/\Z_2\times \R^4$, as in \autoref{Ex_ChargeOne}, and we pick spin structures $\fs$ and $\fu$ as in \autoref{Rmk_CayleySpin}, then
  \begin{equation*}
    \fM
    = (\fs\times\fu\times E_\infty)\times_{\Spin(4)\times G} M
    = (\Re(\Hom(\C^2,\slS^+))\setminus\{0\})/\Z_2 \times NQ.
  \end{equation*}
\end{example}

Denote by $N_\infty Q \coloneq \Fr(NQ)\times_{\SO(4)} S^4$ the sphere-bundle obtained from $NQ$ by adjoining a section at infinity.

\begin{theorem}[Donaldson--Segal \cite{Donaldson2009} and Haydys
  \cite{Haydys2011}]
  \label{Thm_Haydys}
  To each section $\fI\in\Gamma(\fM)$ we can assign a $G$--bundle $E=E(\fI)$ over $N_\infty Q$ together with a connection $I=I(\fI)$ and a framing $f \co E|_{\infty}\to E_\infty$ such that:
  \begin{itemize}
  \item For each $x\in Q$ the restriction of $I$ to $N_xQ$ represents $\fI(x)$.
  \item The framing $f$ identifies the restriction of $I$ to the section at infinity with $A_\infty$.
  \end{itemize}
\end{theorem}

We set $I_\lambda \coloneq I(s_{1/\lambda}^*\fI)$ and impose the condition that
\begin{equation}
  \label{Eq_Limit}
  \lim_{\lambda\to 0} s_\lambda^*\pi_7^0(F_{I_\lambda}) = 0
\end{equation}
where $\pi_7^0$ denotes the zeroth order Taylor expansion of $\pi_7$
off $Q$.
As before, this condition can be phrased in terms of a p.d.e.~on $\fI$.
Define the vertical tangent bundle $V\fM$ to $\fM$ by
\begin{equation*}
  V\fM \coloneq ({\rm Fr}(NQ)\times E_\infty)\times_{\SO(4)\times G} TM.
\end{equation*}
If $\fI$ is a section of $\fM$, then $\Phi$ selects a subbundle
\begin{equation*}
  \Hom_\Phi(TQ,\fI^*V\fM)\subset \Hom(TQ,\fI^*V\fM)
\end{equation*}
and there is a ``Clifford multiplication'' map
\begin{equation*}
  \gamma\co \Hom(TQ,\fI^*V\fM) \to \Hom_\Phi(TQ,\fI^*V\fM)
\end{equation*}
as discussed before.
Moreover, the connections on $NQ$ and $E_\infty$ induce a connection on $\fM$ assigning to each section $\fI$ its covariant derivative $\nabla\fI\in\Omega^1(\fI^*V\fM)$.

\begin{definition}
  The \defined{Fueter operator} $\fF=\fF_\Phi$ associated with $\fM$ is defined by
  \begin{equation*}
    \fI\in\Gamma(\fM) \mapsto \fF_\Phi\fI\coloneq\gamma(\nabla \fI) \in\Gamma(\Hom_\Phi(TQ,\fI^*V\fM)).
  \end{equation*}
  A section $\fI\in\Gamma(\fM)$ is called a \defined{Fueter section} if it satisfies
  \begin{equation*}
    \fF\fI = 0.
  \end{equation*}
\end{definition}

\begin{example}
  \label{Ex_ChargeOneFueter}
  If $M$ is as in \autoref{Ex_ChargeOne}, then the Fueter operator $\fF$ lifts to the twisted Dirac operator
  \begin{equation*}
    \slD\co
    \Gamma(\Re(\Hom(E_\infty,\slS^+)\oplus \slS^+\otimes U))
    \to
    \Gamma(\Re(\Hom(E_\infty,\slS^-)\oplus \slS^-\otimes U)).
  \end{equation*}
\end{example}

The Fueter operator $\fF$ is compatible with the product structure on 
\begin{equation*}
  \fM=\mathring\fM \times NQ
\end{equation*}
corresponding to $M=M^\circ \times\R^4$ with $M^\circ$ denoting the space of instantons centred at zero.
Its restriction to the second factor is given by the Fueter operator $F_Q$ associated with $Q$.

\begin{theorem}[Donaldson--Segal~\cite{Donaldson2009} and Haydys
  \cite{Haydys2011}]
  \label{Thm_Haydys2}
  If $\fI\in\Gamma(\fM)$, then we can identify $\Gamma(\Hom_\Phi(TQ,\fI^*V\fM))$ with a subspace of $\Omega^2\(NQ, \fg_{E(\fI)}\)$.
  With respect to this identification we have the identity
  \begin{equation*}
    \fF\fI = \pi_7^0\(F_{I(\fI)}^{1,1}\).
  \end{equation*}
  In particular, $I(\fI)$ satisfies equation \eqref{Eq_Limit} if and only if $\fI$ is a Fueter section.
\end{theorem}

\begin{definition}
  The \defined{linearised Fueter operator}
  \begin{equation*}
    F_\fI=F_{\fI,\Phi}\co\Gamma(\fI^*V\fM)\to\Gamma(\Hom_\Phi(TQ,\fI^*V\fM))
  \end{equation*}
  for $\fI\in\Gamma(\fM)$ is defined by
  \begin{equation*}
    F_{\fI,\Phi}(\hat \fI)\coloneq\gamma(\nabla \hat \fI)\in \Gamma(\Hom_\Phi(TQ,\fI^*V\fM)).
  \end{equation*}
\end{definition}

\begin{definition}
  \label{Def_FueterUnobstructed}
  A Fueter section $\fI$ is called \defined{unobstructed} if the linearised Fueter operator $F_\fI$ is surjective.
\end{definition}

\begin{example}
  \label{Ex_ChargeOneLinearisedFueter}
  If $M$ is as in \autoref{Ex_ChargeOne}, then the linearised Fueter operator $F_\fI$ lifts to the twisted Dirac operator
  $\slD\co
  \Gamma(\Re(\Hom(E_\infty,\slS^+)\oplus \slS^+\otimes U))
  \to
  \Gamma(\Re(\Hom(E_\infty,\slS^-)\oplus \slS^-\otimes U))$.
  In particular, it only depends on the spin structure $\fs$ and not on $\fI$.
  Using the Atiyah--Singer index theorem we can compute that in the current situation
  \begin{equation}
    \label{Eq_FueterIndex}
    \ind \mathring F_\fI = -\frac14\sigma(Q) - \int_Q c_2(E_\infty)
  \end{equation}
  where $\mathring F_\fI$ is the restriction of $F_\fI$ to $V\mathring \fM$.
\end{example}


%% file: approx.tex
\section{Approximate \texorpdfstring{$\Spin(7)$}{Spin(7)}--instantons}
\label{Sec_Approx}

Throughout the next three sections we assume the hypotheses of \autoref{Thm_A}.
For each sufficiently small gluing parameter $\lambda > 0$ we first construct a connection $A_\lambda$ by grafting $I_\lambda = I(\fI_\lambda)$ into $A_0$ by hand.
$A_\lambda$ will not quite be a $\Spin(7)$--instanton;
however, $\pi_7(F_{A_\lambda})$, the failure of being a $\Spin(7)$--instanton, can be made very small.
We are then left with solving the mildly non-linear p.d.e.
\begin{equation}
  \label{Eq_A}
  \(\rd_{A_\lambda}^* a, \pi_7(F_{A_\lambda+a})\) = L_\lambda a + Q(a) + \pi_7(F_{A_\lambda}) = 0
\end{equation}
with
\begin{equation*}
  L_\lambda
  \coloneq L_{A_\lambda} 
  =
  \begin{pmatrix}
    \rd_{A_\lambda}^* \\
    \pi_7\rd_{A_\lambda}
  \end{pmatrix},
\end{equation*}
see \eqref{Eq_Spin7InstantonLinearisation}, and
\begin{equation*}
  Q(a)\coloneq \frac12\pi_7([a\wedge a])
\end{equation*}
for $a = a(\lambda) \in \Omega^1(X,\fg_{E_\lambda})$.
Given suitable control on $L_\lambda$ and $Q$, \eqref{Eq_A} can be solved by appealing to Banach's fixed-point theorem.

\begin{remark}
  \label{Rmk_Reducible}
  If $A_0$ is reducible, we might not be able to construct $a$ such that $\rd_{A_\lambda}^* a = 0$ on the nose, but only ``modulo $H^0_{A_0}$''.
  For the purpose of proving \autoref{Thm_A} it is not important to have $\rd_{A_\lambda}^* a = 0$.
  If $A_0$ is reducible then, in order to achieve surjectivity, one has to work with $\bar L_\lambda \co \Omega^1(X,\fg_{E_\lambda}) \oplus H^0_{A_0} \to \Omega^0(X,\fg_{E_\lambda}) \oplus \Omega^2_7(X,\fg_{E_\lambda})$ defined by
  \begin{equation*}
    \bar L_\lambda (a,o) = L_\lambda (a) + \iota_\lambda(o)
  \end{equation*}
  where $\iota_\lambda \co H^0_{A_0} \to \Omega^0(X,\fg_{E_\lambda})$ is a inclusion map constructed by first cutting of $o$ to zero near $Q$ and then thinking of it as a section of $\fg_{E_\lambda}$.
  In order to not clutter the exposition any further, we assume in the following that $A_0$ is irreducible.
\end{remark}

\begin{convention}
  We fix a constant $\Lambda>0$ such that all of the statements of the
  kind ``if $\lambda\in(0,\Lambda]$, then \ldots'' appearing in the
  following are valid.  This is possible since there are only a finite
  number of these statements and each one of them is valid provided
  $\Lambda$ is sufficiently small.  By $c>0$ we will denote a generic
  constant whose value does not depend on $\lambda\in(0,\Lambda]$ but
  may change from one occurrence to the next.
\end{convention}

\subsection{Pregluing construction}

\begin{construction}\label{prop:s7:pregluing}
  For each $\lambda\in(0,\Lambda]$ we construct a $G$--bundle $E_{\lambda}$ together with a connection $A_{\lambda}=A\#_\lambda\fI$ from $E_0$, $A_0 \in \sA(E_0)$ and $\fI$.
  The bundles $E_{\lambda}$ are pairwise isomorphic.
\end{construction}

Let us set up some notation.
Fix a constant $\zeta>0$ such that the exponential map identifies a tubular neighbourhood of width $10\zeta$ of $Q$ in $X$ with a neighbourhood of the zero section in $NQ$.
For $I\subset\R$ we set
\begin{equation*}
  U_{I} \coloneq \{ v\in NQ : |v|\in I \}
  \qandq
  V_{I} \coloneq \{ x \in X : r(x)\in I \}.
\end{equation*}
Here
\begin{equation*}
  r\coloneq d(\cdot,Q)\co X\to [0,\infty)
\end{equation*}
 denotes the distance from $Q$.
Fix a smooth-cut off function $\chi\co[0,\infty)\to[0,1]$ which vanishes on $[0,1]$ and is equal to one on $[2,\infty)$.
For $\lambda\in(0,\Lambda]$ we define $\chi^-_{\lambda}\co X\to[0,1]$ and $\chi^+\co X\to[0,1]$ by
\begin{equation*}
  \chi^-_{\lambda}(x)\coloneq \chi(r(x)/\lambda)
  \qandq
  \chi^+(x)\coloneq 1-\chi(r(x)/2\zeta),
\end{equation*}
respectively.

Using radial parallel transport we can identify $E(\fI)$ over $U_{(R,\infty)}$ for some $R>0$ with the pullback of $E(\fI)|_{\infty}$ to said region and similarly we can identify $E_0$ over $V_{[0,\zeta)}$ with the pullback of $E_0|_{Q}$.
Hence, via the framing $\Phi$ we can identify $s_{1/\lambda}^*E(\fI)$ with $E_0$ on the overlap $V_{(\lambda,\sigma)}$ for $\lambda \in (0,\Lambda]$.
Patching both bundles via this identification yields $E_{\lambda}$.

To construct a connection on $E_{\lambda}$ note that on the overlap $I_{\lambda}\coloneq s_{1/\lambda}^*I(\fI)$ and $A_0$ can
be written as
\begin{equation*}
  I_\lambda = A_0|_{Q}+i_{\lambda}
  \qandq
  A_0 = A_0|_{Q}+a.
\end{equation*}
Here and in the following, by a slight abuse of notation, we denote by $A_0|_{Q}$ the pullback of $A_0|_{Q}$ to the overlap.
We define $A_{\lambda}$ by interpolating between $I_\lambda$ and
$A$ on the overlap as follows
\begin{equation}\label{eq:s7:atl}
  A_\lambda \coloneq A_0|_{Q}+\chi^-_{\lambda}a +\chi^+ i_{\lambda}.
\end{equation}
This completes the construction.
\qed

\subsection{Weighted Hölder spaces}

In order to quantify to what extent $\pi_7(A_\lambda)$ is small, we introduce certain norms, which are especially adapted to the geometric situation at hand.

\begin{definition}
  \label{Def_GlobalWeightedNorms}
  For $\lambda\in(0,\Lambda]$ we define a family of weight functions $w_{\ell,\delta;\lambda}$ on $X$ depending on two additional parameters $\ell\in\R$ and $\delta\in\R$ as follows
  \begin{align*}
    w_{\ell,\delta;\lambda}(x) \coloneq 
    \begin{cases}
      \lambda^{\delta}(\lambda+r(x))^{-\ell-\delta}
      &\text{if}~r(x)\leq \sqrt\lambda \\
      r(x)^{-\ell+\delta} &\text{if}~r(x) > \sqrt\lambda
    \end{cases}
  \end{align*}
  and set $w_{\ell,\delta;\lambda}(x,y) \coloneq \min\{w_{\ell,\delta;\lambda}(x), w_{\ell,\delta;\lambda}(y)\}$.
  For a Hölder exponent $\alpha\in(0,1)$ and $\ell,\delta\in\R$ we define (semi-)norms
  \begin{align*}
    \|f\|_{L^\infty_{\ell,\delta;\lambda}(U)}
    &\coloneq \|w_{\ell,\delta;\lambda}f\|_{L^\infty(U)}, \\
    [f]_{C^{0,\alpha}_{\ell,\delta;\lambda}(U)}
    &\coloneq \sup_{\stackrel{x\neq y\in U:}{d(x,y)\leq \lambda + \min\{r(x),r(y)\}}}
    w_{\ell-\alpha,\delta;\lambda}(x,y) \frac{|f(x)-f(y)|}{d(x,y)^\alpha} \qand \\
    \|f\|_{C^{k,\alpha}_{\ell,\delta;\lambda}(U)} 
    &\coloneq \sum_{j=0}^k
    \|\nabla^{k} f\|_{L^\infty_{\ell-j,\delta;\lambda}(U)}
    + [\nabla^{k}f]_{C^{0,\alpha}_{\ell-j,\delta;\lambda}}.
  \end{align*}
  Here $f$ is a section of a vector bundle over $U\subset X$ equipped with an inner product and a compatible connection.
  We use parallel transport to compare the values of $f$ at different points.
  If $U$ is not specified, then we take $U=X$.
\end{definition}

We will primarily use these norms for $\fg_{E_\lambda}$--valued tensor fields.

\begin{remark}
  The reader may find the following heuristic useful.
  Let $f$ be a $k$--form on $X$.
  Fix a small ball centred at a point $x\in Q$, identify it with a small ball in $T_xX=T_xQ\oplus N_xQ$ and rescale this ball by a factor $1/\lambda$. 
  Upon pulling everything back to this rescaled ball the weight function  $w_{-k,\delta;\lambda}$ becomes essentially $\lambda^k(1+|y|)^{k-\delta}$, where $y$ denotes the $N_xQ$--coordinate.
  Thus as $\lambda$ goes to zero a uniform bound $\|f_\lambda\|_{L^\infty_{-k,\delta;\lambda}}$ on a family $(f_\lambda)$ of $k$--forms ensures that the pullbacks of $f_\lambda$ decay like $|y|^{-k+\delta}$ in the direction of $N_xQ$.
  At the same time it forces $f_\lambda$ not to blowup at a rate faster than $r^{-k-\delta}$ along $Q$.
  The ``discrepancy'' in the exponents can be seen to be rather natural by considering the action of the inversion $y \mapsto \lambda y/|y|^2$.
\end{remark}

\begin{prop}
  \label{Prop_Multiplication}
  If $(f,g)\mapsto f\cdot g$ is a bilinear form satisfying $|f\cdot
  g|\leq |f||g|$, then
  \begin{equation*}
    \|f\cdot
    g\|_{C^{k,\alpha}_{\ell_1+\ell_2,\delta_1+\delta_2;\lambda}}
    \leq \|f\|_{C^{k,\alpha}_{\ell_1,\delta_1;\lambda}}
    \|g\|_{C^{k,\alpha}_{\ell_2,\delta_2;\lambda}}.
  \end{equation*}
\end{prop}

\begin{proof}
  This follows immediately from the above definition.
\end{proof}

\begin{cor}
  \label{Cor_DeltaNorm}
  If $\delta<0$, then there is a constant $c>0$ which is independent of $\lambda\in(0,\Lambda]$ such that
  \begin{equation*}
    \|f\|_{C^{k,\alpha}_{\ell,\delta;\lambda}}
    \leq c\lambda^{\delta/2} \|f\|_{C^{k,\alpha}_{\ell,0;\lambda}}
    \qandq
    \|f\|_{C^{k,\alpha}_{\ell,0;\lambda}}
    \leq c \|f\|_{C^{k,\alpha}_{\ell,\delta;\lambda}}
  \end{equation*}
\end{cor}

\begin{proof}
  Use $\|1\|_{C^{k,\alpha}_{0,\delta;\lambda}} \leq c\lambda^{\delta/2}$ and $\|1\|_{C^{k,\alpha}_{0,-\delta;\lambda}} \leq c$ for $\delta<0$.
\end{proof}

There are certain components of $\Omega^1(X,\fg_{E_\lambda})$ and $\Omega^2_7(X,\fg_{E_\lambda})$, which need to be treated separately.
The following definition identifies these components.

\begin{definition}
  Define $\mu_\lambda\co\Gamma(\fI^*V\fM)\to\Omega^1(X,\fg_{E_\lambda})$ by
  \begin{equation*}
    \mu_\lambda\hat\fI\coloneq \chi^+s_{1/\lambda}^*\hat\fI
  \end{equation*}
  and
  $\nu_\lambda\co\Gamma\(\Hom_\Phi(TQ,\fI^*V\fM)\)\to\Omega^2_7(X,\fg_{E_\lambda})$ by
  \begin{equation*}
    \nu_\lambda\hat\fT\coloneq \pi_7(\chi^+s_{1/\lambda}^*\hat\fT).
  \end{equation*}
  Here we first identify $\hat\fI\in\Gamma(\fI^*V\fM)$ with an element of $\Omega^1\(NQ,E(\fI)\)$, then view the restriction of its pullback via $s_{1/\lambda}$ to $U_{[0,\sigma)}$ as lying in $\Omega^1(V_{[0,\sigma)},\fg_{E_\lambda})$ and finally extended it to all of $X$ by multiplication with $\chi^+$; similarly we proceed with $\hat\fT$.
  
  Define $\pi_\lambda\co\Omega^1(X,\fg_{E_\lambda}) \to \Gamma(\fI^*V\fM)$ by
  \begin{equation*}
    (\pi_\lambda a)(x)\coloneq \sum_\kappa \int_{N_xQ} \<a,\mu_\lambda\kappa\> \kappa
  \end{equation*}
  and $\sigma_\lambda\co\Omega^2_7(X,\fg_{E_\lambda}) \to \Gamma\(\Hom_\Phi(TQ,\fI^*V\fM)\)$ by
  \begin{equation*}
    (\sigma_\lambda \alpha)(x)\coloneq \sum_\beta \int_{N_xQ}
    \<\alpha,\nu_\lambda\beta\> \beta,
  \end{equation*}
  Here $\kappa$ runs through an orthonormal basis of $V\fM_{\fI(x)}$ with respect to the inner product $\<\mu_\lambda\cdot,\mu_\lambda\cdot\>$ and $\beta$ runs through an orthonormal basis of $\Hom_\Phi\(T_xQ, V\fM_{\fI(x)}\)$ with respect to the inner product $\<\nu_\lambda\cdot,\nu_\lambda\cdot\>$.

  Clearly, $\pi_\lambda \mu_\lambda=\id$ and $\sigma_\lambda \nu_\lambda=\id$;
  hence,
  \begin{equation*}
    \bar\pi_\lambda\coloneq \mu_\lambda\pi_\lambda \qandq
     \bar\sigma_\lambda\coloneq \nu_\lambda\sigma_\lambda
  \end{equation*}
  are projections.
  We denote the complementary projections by
  \begin{equation*}
    \rho_\lambda \coloneq \id-\bar\pi_\lambda \qandq
    \tau_\lambda \coloneq \id-\bar\sigma_\lambda.
  \end{equation*}
\end{definition}

\begin{prop}
  \label{Prop_ProjectionBounds}
  For $\ell\leq-1$ and $\delta\in\R$ such that $\ell+\delta
  \in (-3,-1)$ there is a constant $c>0$ such that for all
  and $\lambda\in(0,\Lambda]$
  we have
  \begin{align*}
    \|\mu_\lambda\hat\fI\|_{C^{0,\alpha}_{\ell,\delta;\lambda}}
      &\leq c\lambda^{-1-\ell} \|\hat\fI\|_{C^{0,\alpha}} \qandq
    \|\pi_\lambda a\|_{C^{0,\alpha}}
      \leq c
      \lambda^{1+\ell-\alpha}\|a\|_{C^{0,\alpha}_{\ell,\delta;\lambda}(V_{[0,\sigma)})}
  \end{align*}
  as well as
  \begin{align*}
    \|\nu_\lambda\hat\fT\|_{C^{0,\alpha}_{\ell,\delta;\lambda}}
      &\leq c\lambda^{-1-\ell} \|\hat\fT\|_{C^{k,\alpha}} \qandq
    \|\sigma_\lambda \alpha\|_{C^{0,\alpha}}
      \leq c
      \lambda^{1+\ell-\alpha}\|\alpha\|_{C^{0,\alpha}_{\ell,\delta;\lambda}(V_{[0,\sigma)})}.
  \end{align*}
  In particular, $\bar\pi_\lambda$, $\rho_\lambda$,
  $\bar\sigma_\lambda$ and $\tau_\lambda$ are bounded by
  $c\lambda^{-\alpha}$ with respect to the
  $C^{0,\alpha}_{\ell,\delta;\lambda}$--norms.
\end{prop}

\begin{proof}
  We only prove the first two estimates; the last two are identical up to a change in notation.
  From \autoref{Prop_ASDDecay} it follows at once that
  \begin{equation*}
    \|s_{1/\lambda}^*\hat\fI\|_{C^{0,\alpha}_{-3,0;\lambda}\(V_{[0,\sigma)}\)}
    \leq c \lambda^2 \|\hat\fI\|_{C^{0,\alpha}}.
  \end{equation*}
  The first inequality thus is a consequence of
  \autoref{Prop_Multiplication} since
  $\|\chi^+_t\|_{C^{0,\alpha}_{3+\ell,\delta;\lambda}} \leq c
  \lambda^{-3-\ell}$ for $\ell+\delta>-3$.  

  To prove the second inequality, note that by
  \autoref{Prop_ASDDecay} for $\kappa\in(V\fM_t)_{\fI_t(x)}$ we
  have $|s_{1/\lambda}^*\kappa|(y) \leq
  c\lambda^2/(\lambda+|y|)^3\|\kappa\|_{L^2}$ and thus
  \begin{align*}
    \int_{N_xQ} \<a,\chi^+s_{1/\lambda}^*\kappa\>
    &\leq c \int_0^{\sqrt\lambda}
    \lambda^{2-\delta}(\lambda+r)^{\ell+\delta-3}r^3 \rd r 
        \cdot \|a\|_{L^\infty_{\ell,\delta;\lambda}}\|\kappa\|_{L^2}\\
    &\qquad +c \int_{\sqrt\lambda}^\sigma \lambda^2
    r^{\ell-\delta}(\lambda+r)^{-3} r^3 \rd r
     \cdot \|a\|_{L^\infty_{\ell,\delta;\lambda}}\|\kappa\|_{L^2} \\
    &\leq c\lambda^{3+\ell} \|a\|_{L^\infty_{\ell,\delta;\lambda}}\|\kappa\|_{L^2}
  \end{align*}
  since $\ell\leq-1$ and $\ell+\delta<-1$.  If $\kappa$ is an element
  of an orthonormal basis of $(V\fM)_{\fI(x)}$ with respect to
  $\<\mu_{\lambda}\cdot,\mu_{\lambda}\cdot\>$, then
  $\|\kappa\|_{L^2}\leq c/\lambda$ since for $\kappa_1,\kappa_2 \in
  (V\fM)_{\fI(x)}$
  \begin{equation*}
    \lambda^2 \<\kappa_1,\kappa_2\>_{L^2} \sim
    \<\chi^+s_{1/\lambda}^*\kappa_1,\chi^+s_{1/\lambda}^*\kappa_2\>_{L^2}
  \end{equation*}
  where $\sim$ means comparable uniformly in $\lambda$.  Therefore,
  \begin{equation*}
    \|\pi_{\lambda} a\|_{L^\infty} \leq c \lambda^{1+\ell}
    \|a\|_{L^\infty_{\ell,\delta;\lambda}}.
  \end{equation*}
  The estimates on the Hölder norms follow by the same kind of
  argument.
\end{proof}

Ultimately, we will be working with the following function spaces.

\begin{definition}
  Denote by $\fX_\lambda$ and $\fY_\lambda$ the Banach spaces
  $C^{1,\alpha}\Omega^1(X,\fg_{E_\lambda})$ and
  $C^{0,\alpha}\Omega^0(X,\fg_{E_\lambda})\oplus
  C^{0,\alpha}\Omega^2_7(X,\fg_{E_\lambda})$
  equipped with the norms
  \begin{align*}
    \|a\|_{\fX_\lambda}
    &\coloneq \lambda^{-\delta/2}\|\rho_\lambda a\|_{C^{1,\alpha}_{-1,\delta;\lambda}}
    +\lambda\|\pi_\lambda a\|_{C^{1,\alpha}}
    \qand \\
    \|(\xi,\alpha)\|_{\fY_\lambda}
    &\coloneq \lambda^{-\delta/2}\|\xi\|_{C^{0,\alpha}_{-2,\delta;\lambda}}
    +\lambda^{-\delta/2}\|\tau_\lambda \alpha\|_{C^{0,\alpha}_{-2,\delta;\lambda}}
    +\lambda\|\sigma_\lambda \alpha\|_{C^{0,\alpha}},
  \end{align*}
  respectively.
  Here we fix $\delta\in(-1,0)$ and $0<\alpha\ll|\delta|$;
  for concreteness, let us take $\delta=-\frac12$ and $\alpha=\frac{1}{256}$.  
\end{definition}

\begin{remark}
  We choose the factor $\lambda^{-\delta/2}$ in view of \autoref{Cor_DeltaNorm}.
\end{remark}

\subsection{Error estimate}

\begin{prop}\label{Prop_ErrorEstimate}
  There exists a constant $c>0$ such that for all $\lambda\in(0,\Lambda]$
  \begin{equation*}
    \|\pi_7(F_{A_\lambda})\|_{C^{0,\alpha}_{-2,0;\lambda}} \leq c \lambda^2;
  \end{equation*}
  in particular,
  \begin{equation*}
    \|\pi_7(F_{A_\lambda})\|_{\fY_\lambda} \leq c \lambda^{2-\alpha}.
  \end{equation*}
\end{prop}

\begin{remark}
  With more work the exponent can be improved from $2-\alpha$ to $2$.
\end{remark}

The proof of this result requires some preparation.

\begin{prop}
  \label{Prop_Pi7TaylorExpansion}
  In the tubular neighbourhood $V_{[0,\zeta)}$ of $Q$ we can write the Taylor expansion of $\pi_7$ in the direction transverse to $Q$ as
  \begin{equation*}
    \pi_7=\pi_7^0+\pi_7^1+\pi_7^{\geq 2}
  \end{equation*}
  where $\pi_7^0$ denotes the zeroth order term,
  $\pi_7^1$ denotes the first order term and vanishes on $\Lambda^- N^*Q$ and 
  $\pi_7^{\geq 2}$ denotes the remainder term;
  moreover, there is a constant $c>0$ which is independent of $\lambda\in(0,\Lambda]$ such that
  \begin{equation*}
    \|\pi_7^0\|_{C^{0,\alpha}_{0,0;\lambda}(V_{[0,\zeta)})}
    +\|\pi_7^1\|_{C^{0,\alpha}_{1,0;\lambda}(V_{[0,\zeta)})}
    + \|\pi_7^{\geq 2}\|_{C^{0,\alpha}_{2,0;\lambda}(V_{[0,\zeta)})}
    \leq c.
  \end{equation*}
\end{prop}

\begin{proof}
  If we pull the identity map of a tubular neighbourhood of $Q$ back to a tubular neighbourhood of the zero section of $NQ$ via the exponential map, then the Taylor expansion of its derivative around $Q$ can be expressed in the splitting $TNQ=\pi_1^*TQ\oplus\pi_2^*NQ$ as
  \begin{equation*}
    (x,y) \mapsto (x, y) + \(\rII_y(x),y\) + O\(|y|^2\)
  \end{equation*}
  where $\rII$ is the second fundamental form of $Q$ in $X$, which we
  think of as a map from $NQ$ to $\End(TQ)$.
  This immediately yields the desired expansion of $\pi_7$ near $Q$ with $\pi_7^1$ vanishing on $\Lambda^- N^*Q$.
\end{proof}

\begin{prop}\label{Prop_CurvatureEstimate}
  There is a constant $c>0$ such that for all $t\in(-T',T')$ and $\lambda\in(0,\Lambda]$ we have
  \begin{align*}
    \Abs*{F_{I_\lambda}^{2,0}-F_{A_0|Q}}_{C^{0,\alpha}_{-2,0;\lambda}(V_{[0,\sigma)})}
    &\leq c\lambda^2,  \\
    \Abs*{F_{I_\lambda}^{1,1}}_{C^{0,\alpha}_{-3,0;\lambda}(V_{[0,\sigma)})} 
    &\leq c \lambda^2 \qand \\
    \Abs*{F_{I_{t,\lambda}}^{0,2}}_{C^{0,\alpha}_{-4,0;\lambda}(V_{[0,\sigma)})}
    &\leq c \lambda^2.
  \end{align*}
\end{prop}

\begin{proof}
  \autoref{Thm_Haydys} asserts that the restriction of $I = I(\fI)$ to the section at infinity agrees with $A_0|_Q$.
  For a local coordinate system $(z_1,\ldots,z_4,w_1,\ldots,w_4)$ based at a point on the section at infinity and with $z_i$ denoting the coordinates along $Q$ and $w_i$ denote transverse coordinates we can write
  \begin{equation*}
    I = A_0|_Q + \sum_{i,j=1}^4 w_i(\xi_{ij} \rd z_j + \eta_{ij} \rd w_j) + O(|w|^2)
  \end{equation*}
  for $\xi_{ij},\eta_{ij} \in \fg$.
  It follows that $F_I^{1,1} = -\sum_{i,j=1}^4 \xi_{ij} \rd z_i \wedge \rd w_j + O(|w|)$.
  However, by \autoref{Prop_UniversalConnection} and \autoref{Prop_ASDDecay}, when viewed from the zero section the curvature component $F_I^{1,1}$ decays like $r^{-3}$.
  This translates into $\xi_{ij} = 0$, and we can write
  \begin{equation}
    \label{Eq_IFromInfty}
    I = A_0|_Q + \sum_{i,j=1}^4 \eta_{ij} w_i\rd w_j + O(|w|^2).
  \end{equation}

  This means that, $F_I^{2,0}-F_{A_0|_Q}$ vanishes to first order along the section at infinity which when viewed from the zero section in $NQ$ means that
  \begin{equation*}
    \left|F_I^{2,0}-F_{A_0|_Q}\right| \leq \frac{c}{1+|w|^2}.
  \end{equation*}
  The first estimate now follows from a simple scaling consideration.

  The last two estimates follow from simple scaling considerations using \autoref{Prop_ASDDecay} and \autoref{Thm_Haydys} together with the fact that the curvature of a finite energy ASD instanton decays at least like $|y|^{-4}$.
\end{proof}

\begin{prop}\label{Prop_ConnectionEstimate}
  There is a constant $c>0$ such that for all $\lambda\in(0,\Lambda]$ we have
  \begin{align*}
    \|i_\lambda\|_{C^{0,\alpha}_{-3,0;\lambda}(V_{(\lambda,\sigma)})}
    +\|\rd_{I_\lambda}i_\lambda\|_{C^{0,\alpha}_{-4,0;\lambda}(V_{(\lambda,\sigma)})}
    &\leq c\lambda^2 \qand \\
    \|a\|_{C^{0,\alpha}_{1,0;\lambda}(V_{[0,\sigma)})}
    +\|\rd_{A_0|_Q} a\|_{C^{0,\alpha}_{0,0;\lambda}(V_{[0,\sigma)})}
    &\leq c.
  \end{align*}
\end{prop}

\begin{proof}
  The first estimate follows from \eqref{Eq_IFromInfty} and a simple scaling consideration, while the last follows from the fact that we put $A_0$ into radial gauge from zero section in $NQ$.
\end{proof}

\begin{proof}[Proof of \autoref{Prop_ErrorEstimate}]
  We proceed in four steps.
  First we estimate an approximation $\tilde e_\lambda$ of
  \begin{equation*}
    e_\lambda \coloneq \pi_7(F_{A_\lambda}).
  \end{equation*}
  Then we estimate the difference $e_\lambda-\tilde e_\lambda$ separately in the three subsets $V_{[0,\lambda)}$, $V_{[\lambda,\sigma/2)}$ and $V_{[\sigma/2,\sigma)}$ constituting $V_{[0,\sigma)}$ which contains the support of $e_\lambda$.

  It will be convenient to use the following shorthand notation
  \begin{equation*}
    \|f\|_{\ell,U} \coloneq \|f\|_{C^{0,\alpha}_{\ell,0;\lambda}(U)}.
  \end{equation*}
  Note that if $(f,g)\mapsto f\cdot g$ is a bilinear map satisfying $|f\cdot g|\leq|f||g|$, then it follows from \autoref{Prop_Multiplication} that $\|f\cdot g\|_{\ell_1+\ell_2,U}\leq \|f\|_{\ell_1,U}\cdot \|g\|_{\ell_2,U}$.

  \setcounter{step}{0}
  \begin{step}
    The term
    \begin{equation*}
      \tilde e_\lambda \coloneq \pi_7\(F_{I_\lambda} - F_{A_0|_Q}\)
    \end{equation*}
    satisfies $\|\tilde e_\lambda\|_{-2,V_{[0,\sigma)}}\leq
    c\lambda^2$.
  \end{step}

  Because of \autoref{Thm_Haydys2}, the fact that $F^{0,2}_{I_\lambda}$ is anti-self-dual and \autoref{Prop_Pi7TaylorExpansion} we can write $\tilde e_\lambda$ on $V_{[0,\sigma)}$ as
  \begin{equation*}
    \pi_7\(F_{I_\lambda}^{2,0} - F_{A|_Q}\)
    + (\pi_7^1+\pi_7^{\geq 2})\(F_{I_\lambda}^{1,1}\)
    + \pi_7^{\geq 2}\(F_{I_\lambda}^{0,2}\).
  \end{equation*}
  Using \autoref{Prop_Pi7TaylorExpansion} and \autoref{Prop_CurvatureEstimate} as well as $\|1\|_{-1,V_{[0,\sigma)}}\leq c$ we estimate $\|\tilde e_\lambda\|_{-2,V_{[0,\sigma)}}$ by
  \begin{align*}
    &\Abs*{F_{I_{\lambda}}^{(2,0)}-F_{A_0|_Q}}_{-2,V_{[0,\sigma)}} 
    \cdot \Abs*{\pi_7}_{0,V_{[0,\sigma)}} \\
    &\qquad
    +\Abs*{F_{I_\lambda}^{1,1}}_{-3,V_{[0,\sigma)}}
    \cdot\(\Abs*{\pi_7^1}_{1,V_{[0,\sigma)}}+\|1\|_{-1,V_{[0,\sigma)}}
    \cdot \Abs*{\pi_7^{\geq 2}}_{2,V_{[0,\sigma)}}\) \\
    &\qquad
    + \Abs*{F_{I_\lambda}^{0,2}}_{-4,V_{[0,\sigma)}}
    \cdot\|\pi_7^{\geq 2}\|_{2,V_{[0,\sigma)}}
    \leq c \lambda^2.
  \end{align*}

   \begin{step}
     We prove that $\Abs*{e_\lambda-\tilde e_\lambda}_{V_{[0,2\lambda)}} \leq c\lambda^2$.
  \end{step}

  Since
  \begin{equation*}
    \Abs*{\pi_7(F_{A_0|_Q})}_{-2,V_{[0,2\lambda)}}
      \leq \|1\|_{-2,V_{[0,2\lambda)}} \cdot \Abs*{\pi_7(F_{A_0|_Q})}_{0,V_{[0,2\lambda)}} 
      \leq c\lambda^2,
  \end{equation*}
  it suffices to estimate $F_{A_\lambda} - F_{I_\lambda}$ in $V_{[0,2\lambda)}$.
  Now, in $V_{[0,2\lambda)}$ the curvature of $A_{\lambda}$ is given by
  \begin{equation*}
    F_{A_\lambda}
    = F_{I_\lambda}
    + \chi^-_\lambda \rd_{I_\lambda} a
    + \frac12 (\chi^-_\lambda)^2 [a\wedge a]
    + \rd \chi^-_\lambda \wedge a.
  \end{equation*}
  Using \autoref{Prop_ConnectionEstimate} and the fact that the cut-off functions $\chi_\lambda^-$ where constructed so that $\|\chi^-_\lambda\|_{0,V_{[0,\sigma)}} + \|\rd\chi^-_\lambda\|_{-1,V_{[0,\sigma)}} \leq c$ we obtain
  \begin{align*}
    &\|F_{A_\lambda}-F_{I_\lambda}\|_{-2,V_{[0,2\lambda)}} \\
    &\qquad
    \leq 
       \|1\|_{-2,V_{[0,2\lambda)}} \cdot
       \|\chi^-_\lambda\|_{0,V_{[0,2\lambda)}} \cdot
       \|\rd_{A|_Q} a\|_{0,V_{[0,2\lambda)}} \\
    &\qquad\quad 
    +\|\chi^-_\lambda\|_{0,V_{[0,2\lambda)}} \cdot
     \|i_\lambda\|_{-3,V_{[\lambda,\sigma)}} \cdot
     \|a\|_{1,V_{[0,2\lambda)}} \\
    &\qquad\quad
    +\frac12
     \|1\|_{-4,V_{[0,2\lambda)}} \cdot 
     \|\chi^-_{\lambda}\|_{0,V_{[0,2\lambda)}}^2 \cdot
     \|a\|_{1,V_{[0,2\lambda)}}^2     \\
    &\qquad\quad
    +\|1\|_{-2,V_{[0,2\lambda)}} \cdot
    \|\rd\chi^-_\lambda\|_{-1,V_{[0,2\lambda)}} \cdot
     \|a\|_{1,V_{[0,2\lambda)}} 
    \leq c \lambda^2.
  \end{align*}
 
  \begin{step}
    We prove that $\Abs*{e_\lambda -\tilde e_\lambda}_{V_{(2\lambda,\sigma/2)}} \leq c\lambda^2$.
  \end{step}

  This is an immediate consequence of $\pi_7(F_{A_0}) = 0$ and \autoref{Prop_ConnectionEstimate} since in $V_{[2\lambda,\sigma/2)}$ the curvature of $A_\lambda$ is given by $F_{A_\lambda} = F_{A_0} + [i_\lambda \wedge a] + F_{I_\lambda} - F_{A_0|_Q}$.

  \begin{step}
    We prove that $\Abs*{e_\lambda -\tilde e_\lambda}_{V_{[\sigma/2,\sigma)}} \leq c\lambda^2$.
  \end{step}

  In $V_{[\sigma/2,\sigma)}$ the curvature of $A_\lambda$ is given by
  \begin{equation*}
    F_{A_\lambda}
    = F_{A_0}
    + \chi^+ \rd_{A_0} i_\lambda
    + \frac12(\chi^+)^2 [i_\lambda \wedge i_\lambda]
    + \rd \chi^+ \wedge i_\lambda.
  \end{equation*}
  Since $\|\chi^+\|_{\ell,V_{[\sigma/2,\sigma)}} + \|\rd\chi^+\|_{\ell,V_{[\sigma/2,\sigma)}} \leq c$, it follows that
  \begin{align*}
    &\|F_{A_\lambda}-F_{A_0}\|_{-2,V_{[\sigma/2,\sigma)}} \\
    &\qquad
    \leq 
      \|\chi^+\|_{2,V_{[\sigma/2,\sigma)}} \cdot
      \|\rd_{I_\lambda}i_\lambda\|_{-4,V_{[\sigma/2,\sigma)}}\\
    &\qquad\quad
    +\|\chi^+\|_{0,V_{[\sigma/2,\sigma)}} \cdot
     \|a\|_{1,V_{[\sigma/2,\sigma)}} \cdot
     \|i_\lambda\|_{-3,V_{[\sigma/2,\sigma)}}\\
    &\qquad\quad
    +\frac12\|\chi^+\|_{2,V_{[\sigma/2,\sigma)}}^2 \cdot
     \|i_\lambda\|_{-3,V_{[\sigma/2,\sigma)}}^2\\
    &\qquad\quad
    +\|\rd\chi^+\|_{1,V_{[\sigma/2,\sigma)}} \cdot
    \|i_\lambda\|_{-3,V_{[\sigma/2,\sigma)}} \leq c\lambda^2.
  \end{align*}
  This completes the estimate.
\end{proof}


%% file: la.tex
\section{Linear analysis}
\label{Sec_LinearAnalysis}

\begin{prop}
  \label{Prop_RightInverse}
  For $\lambda\in(0,\Lambda]$ the linear operator $L_\lambda\co \fX_\lambda \to \fY_\lambda$ has a right inverse $R_\lambda\co \fY_\lambda \to \fX_\lambda$ and there exists a constant $c>0$ which is independent of $\lambda\in(0,\Lambda]$ such that
  \begin{equation*}
    \|R_\lambda (\xi,\alpha)\|_{\fX_\lambda} \leq c \|(\xi,\alpha)\|_{\fY_\lambda}.
  \end{equation*}
\end{prop}

This is the key to proving \autoref{Thm_A}.
We produce $R_\lambda$ by gluing various local right inverses ``by hand''.
We decompose $L_\lambda$ as
\begin{align*}
  L_\lambda
  =
  \begin{pmatrix}
    \fK_\lambda & \fp_\lambda \\
    \fq_\lambda & \fL_\lambda
  \end{pmatrix}
\end{align*}
where
\begin{gather*}
  \fK_\lambda \coloneq \bar\sigma_\lambda L_\lambda \bar\pi_\lambda,  \quad
  \fL_\lambda \coloneq \tau_\lambda L_\lambda \rho_\lambda,  \\
  \fp_\lambda \coloneq \bar\sigma_\lambda L_\lambda \rho_\lambda, \qandq
  \fq_\lambda \coloneq \tau_\lambda L_\lambda \bar\pi_\lambda.
\end{gather*}
In the course of this section we will show that $\fK_\lambda$ is essentially the linearised Fueter operator $F_\fI$, which has a right inverse by assumption, and that local right inverses for $\fL_\lambda$ can be seen to exist by considerations of model operators on $\R^8$ and on the complement of $Q$, while $\fp_\lambda$ and $\fq_\lambda$ are negligibly small terms.
An approximate right inverse $\tilde R_\lambda$ can then be constructed by carefully patching together the local right inverses.
Finally, a simple deformation argument will yield $R_\lambda$.

\subsection{The model operator on \texorpdfstring{$\R^8$}{R8}}
\label{Sec_R8Model}

Let $I$ be a finite energy ASD instanton on a $G$--bundle $E$ over $\R^4$.
By a slight abuse of notation we denote the pullbacks of $I$ and $E$ to $\R^8=\R^4\times \R^4$ by $I$ and $E$ as well.
We define $\bL_I\co \Omega^0(\R^8,\fg_E) \to \Omega^0(\R^8,\fg_E)\oplus\Omega^2_7(\R^8,\fg_E)$ by
\begin{equation*}
  \bL_I(a)\coloneq(\rd_A^*a,\pi_7\rd_A a).
\end{equation*}
Here $\pi_7$ is taken with respect to the standard $\Spin(7)$--structure $\Phi_0$ on $\R^8$, see \eqref{Eq_Phi0}.

By \autoref{Rmk_Spin7InstantonG2xLine} we can, with the appropriate identifications being
made, write
\begin{equation*}
  \bL_I=\del_t-L_I
\end{equation*}
where we think of $I$ as a $\rG_2$--instanton on
$\{0\}\times \R^3\times \R^4$ and $L_I$ is as in 
\begin{equation*}
    L_{A,\phi} \coloneq
    \begin{pmatrix}
      0 & \rd_A^* \\
      \rd_A & *\left(\psi\wedge\rd_A\right)
  \end{pmatrix}.
\end{equation*}
In particular, using \cite[Proposition 7.1]{Walpuski2011} we see that
\begin{align}
  \bL_I\bL_I^*=\bL_I^*\bL_I = \Delta_{\R^4}  +
  \begin{pmatrix}
    \delta_I\delta_I^* & \\
    & \delta_I^*\delta_I
  \end{pmatrix}
\end{align}
and, hence, we can argue as in \cite[Section 7]{Walpuski2011}.

\begin{remark}
  In the above situation thinking of $\R^8$ as $\R^4\times \R^4$ as in \autoref{Ex_HyperkahlerPair} and at the same time as $\R\times (\R^3\times \R^4)$ as in \autoref{Ex_G2xLine}, the summands $\Lambda^2_3$ and $\Lambda^2_4$ in \autoref{Prop_Lambda2SplittingHyperkahlerPair} are identified, via \autoref{Prop_Lambda2SplittingG2xLine}, with $\R^3$ and $\R^4$ respectively.
\end{remark}

\begin{definition}
  \label{Def_R8WeightedNorms}
  Define weight functions $w \co \R^8 \to [0,\infty)$ and, by slight abuse of notation, $w \co (\R^8)^2 \to [0,\infty)$ by
  \begin{equation*}
    w(x) \coloneq 1 + |\pi_2(x)| \quad\text{and}\quad
    w(x,y) \coloneq \min\{w(x),w(y)\}.
  \end{equation*}
  Here $\pi_2 \co \R^8 = \R^4\times \R^4 \to \R^4$ is the projection to the second factor.
  For a Hölder exponent $\alpha\in(0,1)$ and a weight parameter $\beta\in\R$ we define
  \begin{align*}
    [f]_{C^{0,\alpha}_{\beta}(U)}
    &\coloneq \sup_{d(x,y) \leq w(x,y)}
    w(x,y)^{\alpha-\beta} \frac{|f(x)-f(y)|}{d(x,y)^\alpha}, \\[1ex]
    \|f\|_{L^{\infty}_{\beta}(U)}
    &\coloneq\big\|w^{-\beta}f\big\|_{L^\infty(U)} \qand \\
    \|f\|_{C^{k,\alpha}_{\beta}(U)}
    &\coloneq \sum_{j=0}^k \big\|\nabla^j f\big\|_{L^{\infty}_{\beta-j}(U)}
    + \big[\nabla^j f\big]_{C^{0,\alpha}_{\beta-j}(U)}.
  \end{align*}
  Here $f$ is a section of a vector bundle over $U\subset\R^8$ equipped with an inner product and a compatible connection.
  We use parallel transport to compare the values of $f$ at different points.
  If $U$ is not specified, then we take $U=\R^8$.
  We denote by $C^{k,\alpha}_\beta$ the subspace of elements $f$ of the Banach space $C^{k,\alpha}$ with $\smash{\|f\|_{C^{k,\alpha}_\beta}}<\infty$ and equip it with the norm $\smash{\|\cdot\|_{C^{k,\alpha}_\beta}}$.
\end{definition}

The linear operator $\bL_I$ can serve as a model for $L_\lambda$ in the following sense:
Fix $x\in Q$.
Set $I\coloneq I(\fI)|_{N_xQ}$ and $E\coloneq E(\fI)|_{N_xQ}$.
Identify $T_xX=T_xQ\times N_xQ$ with $\R^8=\R^4\times\R^4$ in such a way that the summands are preserved and $\Phi|_{T_xX}$ is identified with $\Phi_0$.
For $\epsilon_1,\epsilon_2>0$ we define
\begin{equation*}
  V_{\epsilon_1,\epsilon_2}\coloneq B_{\epsilon_1}(x)\cap V_{[0,\epsilon_2)}.
\end{equation*}
Using the exponential map we can identify $V_{\epsilon_1,\epsilon_2}$ with a small neighbourhood $\tilde U_{\epsilon_1,\epsilon_2}$ of the origin in $\R^8$.
With respect to this identification a $\fg_{E_\lambda}$--valued tensor field $f$ on $V_{\epsilon_1,\epsilon_2}$ is identified with a $s_{1/\lambda}^*\fg_E$--valued tensor field $\tilde f$ on $\tilde U_{\epsilon_1,\epsilon_2;\lambda}$, and if $k\in\N$ is a scaling parameter, then with $f$ we can associate a $\fg_E$--valued tensor field $s_{d,\lambda}f$ on
\begin{equation*}
  U_{\epsilon_1,\epsilon_2;\lambda}
  \coloneq
  \lambda^{-1}\tilde U_{\epsilon_1,\epsilon_2}
  =
  \lambda^{-1}\exp_x^{-1}(V_{\epsilon_1,\epsilon_2})
\end{equation*}
defined by
\begin{equation*}
  (s_{d,\lambda} f)(x,y)
  \coloneq
  \lambda^d \tilde f(\lambda x, \lambda y)
  =
  \lambda^d f\circ\exp(\lambda (x, y)).
\end{equation*}

\begin{prop}
  \label{Prop_LModelComparison}
  There are constants $c,\epsilon_0>0$ such that for
  $\epsilon\in(0,\epsilon_0]$ and $\lambda\in(0,\Lambda]$ we have
  \begin{multline*}
    \frac{1}{c}\lambda^{d+\ell} \|f\|_{C^{k,\alpha}_{\ell,\delta;\lambda}\(V_{\epsilon,N\sqrt\lambda}\)}
    \leq
    \|s_{d,\lambda}f\|_{C^{k,\alpha}_{\ell+\delta}\(U_{\epsilon,N\sqrt\lambda;\lambda}\)} \\
    \leq cN^{-2\delta} 
    \lambda^{d+\ell}\|f\|_{C^{k,\alpha}_{\ell,\delta;\lambda}\(V_{\epsilon,N\sqrt\lambda}\)}    
  \end{multline*}
  and
  \begin{equation*}
    \Abs*{L_\lambda a - s_{2,\lambda}^{-1} \bL_I s_{1,\lambda} a}_{C^{0,\alpha}_{-2,\delta;\lambda}\(V_{\epsilon,N\sqrt\lambda}\)}
    \leq c(\epsilon+\sqrt\lambda)
    \Abs*{a}_{C^{1,\alpha}_{-1,\delta;\lambda}\(V_{\epsilon,N\sqrt\lambda}\)}.
  \end{equation*}
  Here, in the first estimate, we also allow $k = \alpha = 0$, thus making a statement about weighted $L^\infty$--norms.
\end{prop}

For $\beta<-1$ we define $\pi_I\co C^{k,\alpha}_{\beta} \to C^{k,\alpha}(\R^4,\ker\delta_I)$ by
\begin{equation*}
  \pi_I(a)(x)\coloneq\sum_\kappa\<a(x,\cdot),\kappa\>_{L^2(\R^4)}\kappa
\end{equation*}
where $\kappa$ runs through an $L^2$ orthonormal basis of $\ker \delta_I$ and set
\begin{equation*}
  \fA^{k,\alpha}_{\beta}\coloneq\ker \pi_I \cap C^{k,\alpha}_{\beta}.
\end{equation*}
The projection operators $\pi_\lambda$ and $\sigma_\lambda$ can be viewed as ``global versions'' of $\pi_I$.  It follows from the
discussion following \autoref{Prop_ASDDecay} that $\bL_I$ defines a linear $\bL_I\co\fA^{1,\alpha}_{\beta}\to\fA^{0,\alpha}_{\beta-1}$.

The key result of this section is the following.
\begin{prop}\label{Prop_R8ModelInvertible}
  For $\beta\in(-2,-1)$ the linear operator $\bL_I \co \fA^{1,\alpha}_{\beta} \to \fA^{0,\alpha}_{\beta-1}$ is invertible.
\end{prop}

The proof rests on the following estimate.

\begin{prop}\label{Prop_R8ModelEstimates}
  For $\beta\in(-3,-1)$ there is a constant $c>0$ such that for all $a \in \fA^{1,\alpha}_\beta$ the following holds
  \begin{equation*}
    \|a\|_{C^{1,\alpha}_\beta} \leq c \|\bL_I a\|_{C^{0,\alpha}_{\beta-1}} \qandq
    \|a\|_{C^{1,\alpha}_\beta} \leq c \|\bL_I^* a\|_{C^{0,\alpha}_{\beta-1}}.
  \end{equation*}
\end{prop}

\begin{proof}[Proof of \autoref{Prop_R8ModelInvertible} assuming \autoref{Prop_R8ModelEstimates}]
  From \autoref{Prop_R8ModelEstimates} it follows that $\bL_I \co \fA^{1,\alpha}_\beta \to \fA^{0,\alpha}_{\beta-1}$ is injective and its image is closed.
  Thus we can identify its cokernel with the kernel of $\bL_I^* \co \bigl(\fA^{0,\alpha}_{\beta-1}\bigr)^* \to \bigl(\fA^{1,\alpha}_\beta\bigr)^*$.
  Since $\beta > -2$, the image of $\pi_I$ is contained in $C^{0,\alpha}_{\beta-1}$ and thus $C^{0,\alpha}_{\beta-1}=\fA^{0,\alpha}_{\beta-1} \oplus \im \pi_I$.
  Via this splitting we can extend any $b \in \ker \bL_I^*$ to an element of $\bigl(C^{0,\alpha}_{\beta-1}\bigr)^*$ which still satisfies $\bL_I^* b = 0$.
  By elliptic regularity $b$ is smooth and it follows from \autoref{Lem_Liouville} that $b$ is invariant under translations in the $\R^4$--direction.
  Now, $b$ must be contained in $C^{1,\alpha}_{-3-\beta}$.
  Since $-3-\beta\in(-3,-1)$, it follows that $b=0$ by \autoref{Prop_R8ModelEstimates}.
  Therefore $\bL_I$ is also surjective; hence, invertible.
\end{proof}

\begin{lemma}[{\cite[Lemma A.1]{Walpuski2011}}]
  \label{Lem_Liouville}
  Let $E$ be a vector bundle of bounded geometry over a Riemannian manifold $X$ of bounded geometry and with subexponential volume growth, and suppose that $D\co C^\infty(X,E)\to{}C^\infty(X,E)$ is a uniformly elliptic operator of second order whose coefficients and their first derivatives are uniformly bounded, that is non-negative, such that $\<Da,a\>\geq 0$ for all $a \in W^{2,2}(X,E)$, and formally self-adjoint.
  If $a \in C^\infty(\R^n\times X,E)$ satisfies
  \begin{equation*}
    (\Delta_{\R^n}+D)a = 0
  \end{equation*}
  and $\|a\|_{L^\infty}$ is finite, then $a$ is constant in the $\R^n$--direction, that is $a(x,y)=a(y)$.
  Here, by slight abuse of notation, we denote the pullback of $E$ to $\R^n\times X$ by $E$ as well.
\end{lemma}

\begin{proof}[Proof of \autoref{Prop_R8ModelEstimates}]
  We restrict to the case of $\bL_I$ as the case $\bL_I^*$ differs only by a slight change in notation.
  First, it is easy to see that there are Schauder estimates, cf.~\cite[Proposition 7.6]{Walpuski2011},
  \begin{equation*}
    \|a\|_{C^{1,\alpha}_\beta}
    \leq c\(\|\bL_I a\|_{C^{0,\alpha}_{\beta-1}}
        + \|a\|_{L^\infty_\beta}\)
  \end{equation*}
  with $c=c(\beta)>0$.
  The crucial step is then to show that if $\beta\in(-3,-1)$ there is a constant $c>0$ such that for all $a\in\fA^{1,\alpha}_\beta$ we have
  \begin{equation*}
    \|a\|_{L^\infty_\beta}
    \leq c\|\bL_I a\|_{C^{0,\alpha}}.
  \end{equation*}
  This is proved by contradiction:
  Suppose the estimate does not hold.
  Then there exists a sequence $a_i\in\fA^{1,\alpha}_\beta$ such that
  \begin{equation*}
    \|a_i\|_{L^{\infty}_\beta}=1
    \qandq
    \|\bL_I a_i\|_{C^{0,\alpha}_{\beta-1}} \leq \frac1i.
  \end{equation*}
  Hence, by the above Schauder estimate
  \begin{equation*}
    \|a_i\|_{C^{1,\alpha}_\beta} \leq 2c.
  \end{equation*}
  Pick $(x_i,y_i) \in \R^4\times\R^4$ such that
  \begin{equation*}
    w(x_i,y_i)^{-\beta}|a_i(x_i,y_i)|=1.
  \end{equation*}
  By translation we can assume that $x_i=0$.
  Without loss of generality one of the following two cases must occur.
  We rule out both of them thus proving the estimate.

  \setcounter{case}{0}
  \begin{case}
    The sequence $|y_i|$ stays bounded.
  \end{case}

  Let $K$ be a compact subset of $\R^8$.
  When restricted to $K$, the elements $a_i$ are uniformly bounded in $C^{1,\alpha}$.
  Thus, by Arzelà--Ascoli, we can assume (after passing to a subsequence) that $a_i$ converges to a limit $a$ in $C^{1,\alpha/2}$.
  Since $K$ was arbitrary, this yields $a \in \Omega^1(\R^8,\fg_E)$ satisfying
  \begin{equation*}
    |a|(x,y)< c (1+|y|)^{\beta}
  \end{equation*}
  as well as 
  \begin{equation*}
    \bL_I a=0 \quad\text{and}\quad
    \pi_I a=0.
  \end{equation*}
  It follows from \autoref{Lem_Liouville} that $a=0$.
  On the other hand we can assume that $y_i$ converges to some point $y \in \R^4$ for which we would have $|a|(0,y)=w(0,y)^{\beta} \neq 0$.
  This is a contradiction.

  \begin{case}
    The sequence $|y_i|$ goes to infinity.
  \end{case}

  Define a rescaled sequence $\tilde a_i$ by
  \begin{equation*}
     \tilde a_i(x,y)\coloneq|y_i|^{-\beta}(\xi_i,a_i)(|y_i|x,|y_i|y)
  \end{equation*}
  and set $\tilde y_i=y_i/|y_i|$.
  The rescaled
  sequence then satisfies
  \begin{gather*}
    \|\tilde a_i\|_{\tilde C^{1,\alpha}_\beta} \leq 2c, \quad
    \|\bL\tilde a_i\|_{\tilde C^{0,\alpha}_{\beta-1}} \leq
    2/i \quad\text{and}\quad
    \tilde w(0,\tilde y_i)^{-\beta}|\tilde a_i(0,\tilde y_i)| \geq 1/2 
  \end{gather*}
  where the norms $\|\cdot\|_{\tilde C^{k,\alpha}_\beta}$ are defined as those in \autoref{Def_R8WeightedNorms}, but with weight function $w(x)=|\pi_2(x)|$ instead of $w(x)=1+|\pi_2(x)|$, and where $L$ is defined by
  \begin{equation*}
    \bL\coloneq\del_t - L
  \end{equation*}
  with
  \begin{equation*}
    L(\xi,a)\coloneq\(\rd^*a, \rd \xi+*(\psi\wedge\rd a)\).
  \end{equation*}

  We can now pass to a limit using Arzelà--Ascoli as before to obtain $\tilde a$ defined over $\R^4\times\(\R^4\setminus\{0\}\)$ satisfying
  \begin{equation*}
    |\tilde a|(x,y) < c |y|^\beta \quad\text{and}\quad L \tilde a=0.
  \end{equation*}
  Since $\beta>-3$, $L\tilde a=0$ holds on all of $\R^8$ in the sense of distributions.
  Hence, by standard elliptic theory, $\tilde a$ extends to a bounded smooth solution of $L\tilde a=0$ on $\R^8$.
  Since $L^*L=\Delta_{\R^4}+\Delta_{\R^4}$, it follows from \autoref{Lem_Liouville} that $\tilde a$ is invariant in the $\R^4$--direction.
  Therefore, we can think of the components of  $\tilde a$ as harmonic functions on $\R^4$.
  These decay to zero at infinity as $\beta<0$ and, hence, must vanish identically.
  On the other hand we know that $|\tilde y_i|=1$ and thus without loss of generality $\tilde y_i$ converges to some point $\tilde y$ in the unit sphere for which $|\tilde a|(0,\tilde y)|\geq \frac12$, contradicting $\tilde a=0$.
\end{proof}

\subsection{The model away from \texorpdfstring{$Q$}{Q}}
\label{Sec_XminusQModel}

\begin{definition}
  Define weighted Hölder norms $\|\cdot\|_{C^{k,\alpha}_{\beta}}$ for tensor fields (with values in $\fg_E$) on $X\setminus Q$ by
  \begin{align*}
    [f]_{C^{0,\alpha}_{\beta}}
    &\coloneq \sup_{d(x,y) \leq w(x,y)}
    w(x,y)^{\alpha-\beta} \frac{|f(x)-f(y)|}{d(x,y)^\alpha}. \\
    \|f\|_{L^{\infty}_{\beta}}
    &\coloneq\|w^{-\beta}f\|_{L^\infty} \qand \\
    \|f\|_{C^{k,\alpha}_{\beta}}
    &\coloneq \sum_{j=0}^k \|\nabla^j f\|_{L^{\infty}_{\beta-j}}
    + [\nabla^j f]_{C^{0,\alpha}_{\beta-j}}.
  \end{align*}
  with weight functions given by
  \begin{equation*}
    w(x) \coloneq r(x) \quad\text{and}\quad w(x,y) \coloneq \min\{w(x),w(y)\}.
  \end{equation*}
  (Recall, that $r \co X \to [0,\infty)$ is defined by $r(x) = d(\cdot, Q)$.
\end{definition}

If we fix a constant $N>0$, then over $V_{[\sqrt\lambda/N,\infty)}$ we can view a tensor field $f$ with values in $\fg_{E_\lambda}$ as one which takes values in $\fg_{E}$ and vice versa.

\begin{prop}\label{Prop_XminusQModelComparison}
  There is a constant $c>0$ such that for
  $\lambda\in(0,\Lambda]$ with respect to the above identification we have
  \begin{equation*}
    \frac{1}{c}\|a\|_{C^{k,\alpha}_{-\ell+\delta}\(V_{[\sqrt\lambda/N,\infty)}\)}
     \leq \|a\|_{C^{k,\alpha}_{\ell,\delta,\lambda}\(V_{[\sqrt\lambda/N,\infty)}\)}
     \leq cN^{-2\delta}\|a\|_{C^{k,\alpha}_{-\ell+\delta}\(V_{[\sqrt\lambda/N,\infty)}\)}
   \end{equation*}
   and
   \begin{equation*}
    \|L_\lambda a - L_{A_0} a\|_{C^{0,\alpha}_{-2,\delta,\lambda}\(V_{[\sqrt\lambda/N,\infty)}\)}
    \leq c \sqrt\lambda/N |a\|_{C^{1,\alpha}_{-1,\delta,\lambda}\(V_{[\sqrt\lambda/N,\infty)}\)}.
  \end{equation*}
\end{prop}

\begin{prop}\label{Prop_XminusQInverse}
  For $\beta\in(-3,0)$ the operator $L_{A_0} \co C^{1,\alpha}_{\beta} \to C^{0,\alpha}_{\beta-1}$ has a right inverse $R_{A_0}$.
\end{prop}

\begin{proof}
  Denote by $\pi\co C^{1,\alpha}_\beta\to \ker L_A$ the $L^2$--projection to the (smooth) kernel of $L_A$.
  This is well defined, because $\beta >-3$.
  We will shortly prove the estimates
  \begin{align*}
    \|a\|_{C^{1,\alpha}_\beta} &\leq c
    \(\|L_A a\|_{C^{0,\alpha}_{\beta-1}} + \|\pi a\|_{L^\infty_\beta}\) \\
    \andq
    \|a\|_{C^{1,\alpha}_\beta} &\leq c
    \|L_A^* a\|_{C^{0,\alpha}_{\beta-1}}.
  \end{align*}
  From the first estimate it follows immediately that the image of $L_A \co C^{1,\alpha}_{\beta} \to C^{0,\alpha}_{\beta-1}$ is closed and its kernel is finite-dimensional (in fact, it can be seen to agree with the smooth kernel of $L_A$).
  To show that $L_A$ has a right inverse it suffices to prove that $\coker L_A=0$.
  Let $b\in\ker L_A^* \iso \coker L_A$.
  Then using elliptic regularity it can be seen that $b$ represents an element in the kernel of $L_A^*\co C^{1,\alpha}_{-3-\beta}\to C^{0,\alpha}_{-4-\beta}$.
  But then $b=0$ by the second estimate.
 
  Now we are left with proving the above estimates.
  We will only prove the first estimate, since the proof of the second estimate is similar, but slightly easier.
  First of all we have the following Schauder estimate
  \begin{equation*}
    \|a\|_{C^{1,\alpha}_{\beta,t}}  \leq c
    (\|L_A a\|_{C^{0,\alpha}_{\beta-1,t}} + \|a\|_{L^\infty_{\beta,t}}).
  \end{equation*}
  To prove that
  \begin{equation*}
    \|a\|_{L^\infty_{\beta,t}} \leq c \(\|L_A a\|_{C^{0,\alpha}_{\beta-1,t}} + \|\pi a\|_{L^\infty_{\beta,t}}\)    
  \end{equation*}
  one argues by contradiction.
  If $a_i$ is a sequence of counterexamples as before, then we can assume that it either gives rise to a non-trivial element $a$ in the kernel of $L_A\co C^{1,\alpha}_\beta\to C^{0,\alpha}_{\beta-1}$ which also satisfies $\pi a=0$ or localises in smaller and smaller neighbourhoods of $Q$.
  To see that the first case cannot occur observe that if $a\in C^{1,\alpha}_{\beta}$ solves $L_Aa=0$ on $X\setminus Q$, then it follows that $L_Aa=0$ on all of $X$ in the sense of distributions and thus $a$ extends smoothly to $X$, since $\beta>-3$.
  This contradicts $\pi a=0$.
  Thus we must be in the second case.
  Rescaling $a_i$ near $Q$ as before yields a non-trivial harmonic function on $\R^4\times \R^4 \setminus \{0\}$ which is bounded by a constant multiple of $|y|^\beta$.
  Since $\beta>-3$ the function extends to $\R^8$ and by \autoref{Lem_Liouville} it is invariant in the $\R^4$--direction.
  Hence, it corresponds to a decaying harmonic function on $\R^4$, since $\beta<0$, and must vanish identically.
  So the second case does not occur either; thus proving that the claimed estimate must hold.
\end{proof}

\subsection{Comparison of \texorpdfstring{$\fK_\lambda$}{K} with \texorpdfstring{$F_{\fI}$}{F}}

\begin{prop}\label{Prop_LFComparison}
  There is a constant $c>0$ such that for all $\lambda\in(0,\Lambda]$
  we have
  \begin{equation*}
    \|(L_\lambda\mu_\lambda-\nu_\lambda F_\fI)\hat\fI\|_{C^{0,\alpha}_{-2,0;\lambda}}
       \leq c \lambda^2 \|\hat\fI\|_{C^{1,\alpha}}.
  \end{equation*}
\end{prop}

\begin{cor}\label{Cor_LFComparison}
  There is a constant $c>0$ such that for all $\lambda\in(0,\Lambda]$
  we have
  \begin{equation*}
    \|(\sigma_\lambda L_\lambda\mu_\lambda- F_\fI)\hat\fI\|_{C^{0,\alpha}}
       \leq c \lambda^{1-\alpha} \|\hat\fI\|_{C^{1,\alpha}}.
  \end{equation*}
\end{cor}

\begin{proof}[Proof of \autoref{Prop_LFComparison}]
  We use the model operator $\tilde L_\lambda$ defined by
  \begin{equation*}
    \tilde L_\lambda a\coloneq \(\rd_{I_\lambda}^*a,\pi_7^0(\rd_{I_\lambda}a)\).
  \end{equation*}
  If we view $\Gamma(\fI^*V\fM)$ as a subspace of
  $\Omega^1(NQ,\fg_E)$, then on this subspace $\tilde L_\lambda$
  agrees with the linearised Fueter operator $F_{\fI}$.  We thus
  have to estimate the terms in the expression
   \begin{align*}
    L_\lambda\mu_\lambda\hat\fI
    -\nu_\lambda F_{\fI}\hat\fI
    &=L_\lambda (\mu_\lambda \hat\fI -
    \hat\fI_\lambda) + (L_\lambda-\tilde L_\lambda)\hat\fI
    +  s_{1/\lambda}^*F_\fI\hat\fI-\nu_\lambda F_\fI\hat\fI\\ &=:\rI+\rII+\rIII
  \end{align*}
  on $V_{[0,\zeta)}$.
  It is easy to see that
  \begin{equation*}
    \|\rI\|_{C^{0,\alpha}_{-2,0;\lambda}(V_{[0,\zeta)})}
    +\|\rIII\|_{C^{0,\alpha}_{-2,0;\lambda}(V_{[0,\zeta)})} \leq c\lambda^2
    \|\hat\fI\|_{C^{1,\alpha}}.
  \end{equation*}
  by using that fact that $\rI$ and $\rIII$ are supported in
  $V_{[\sigma/2,\sigma)}$ and the estimates
  \begin{equation*}
    \|L_\lambda a\|_{C^{0,\alpha}_{-2,0;\lambda}(V_{[0,\sigma)})}
    \leq c \|a\|_{C^{1,\alpha}_{-1,0;\lambda}(V_{[0,\sigma)})} 
    \qandq
    \|F_\fI\hat\fI\|_{C^{0,\alpha}}\leq c \|\hat\fI\|_{C^{1,\alpha}}
  \end{equation*}
  as well as
  \begin{align*}
    \|\mu_\lambda\hat\fI
    -\hat\fI_\lambda\|_{C^{k,\alpha}_{-\ell,0,\lambda}(V_{[\sigma/2,\sigma)})}
    &\leq \|\chi^+-1\|_{{C^{k,\alpha}_{\ell+3,0;\lambda}}(V_{[\sigma/2,\sigma)})}
    \cdot\|\hat \fI_\lambda\|_{C^{k,\alpha}_{-3,0;\lambda}(V_{[\sigma/2,\sigma)})} \\
    &\leq c\lambda^2 \|\hat\fI\|_{C^{k,\alpha}}
  \end{align*}
  and a similar estimate for $\nu_\lambda$.

  The key for the estimate of $\rII$ is to notice that
  \begin{equation*}
    \pi_7^0\((\rd_{I_\lambda}\hat\fI)^{0,2}\) =
    \pi_7^1\((\rd_{I_\lambda}\hat\fI)^{0,2}\) = 0,
  \end{equation*}
  because $\delta_{\fI(x)}(\hat\fI|_{N_xQ})=0$ and $\pi_7^0$ and $\pi_7^1$ vanish on $\Lambda^- NQ$.  
  Therefore,
  \begin{align*}
    \rII 
    &= \pi_7\((A_\lambda - I_\lambda)\wedge \hat\fI_\lambda\)
    + \pi_7^1\((\rd_{I_\lambda}\hat\fI_\lambda)^{2,0} + (\rd_{I_\lambda}\hat\fI_\lambda)^{1,1}\)
    + \pi_7^{\geq 2}(\rd_{I_\lambda}\hat\fI_\lambda) \\
    &=: \rII_1 + \rII_2 + \rII_3.
  \end{align*}
  It follows from \autoref{Prop_ConnectionEstimate} that
  \begin{equation}
    \label{eq:aie}
    \|A_\lambda-I_\lambda\|_{C^{0,\alpha}_{1,0;\lambda}(V_{[0,\sigma)})}
    = \|\chi_\lambda^-a + (\chi^+-1)i_\lambda\|_{C^{0,\alpha}_{1,0;\lambda}(V_{[0,\sigma)})}
    \leq c
  \end{equation}
  which in conjunction with 
  \begin{equation}
    \label{Eq_fIComparison}
    \|\hat\fI_\lambda\|_{C^{k,\alpha}_{-3,0;\lambda}(V_{[0,\sigma)})} \leq c
    \lambda^2 \|\hat\fI\|_{C^{k,\alpha}}
  \end{equation}
  yields
  \begin{equation*}
    \|\rII_1\|_{C^{0,\alpha}_{-2,0;\lambda}}
    \leq c\lambda^2\|\hat\fI\|_{C^{1,\alpha}}.
  \end{equation*}
  $\rII_2$ and $\rII_3$ can be estimated using \autoref{Prop_Pi7TaylorExpansion}, \autoref{Prop_ConnectionEstimate} and \eqref{Eq_fIComparison}.
\end{proof}

\subsection{Estimate of \texorpdfstring{$\fp_\lambda$}{p} and \texorpdfstring{$\fq_\lambda$}{q}}

\begin{prop}\label{Prop_OffDiagonalEstimate}
  For $\delta\in(-1,0)$ there exists a constant $c>0$ such that
  for all  $\lambda\in(0,\Lambda]$ we have
  \begin{align*}
    \|\sigma_\lambda\fp_\lambda a\|_{C^{0,\alpha}} 
      &\leq c\lambda^{-\alpha} \|\rho_\lambda a\|_{C^{1,\alpha}_{-1,\delta;\lambda}} \qand \\
    \|\fq_\lambda a\|_{C^{0,\alpha}_{-2,\delta;\lambda}} 
      &\leq c \lambda^{2+\delta/2-\alpha} \|\pi_\lambda a\|_{C^{1,\alpha}}.
  \end{align*}
\end{prop}

\begin{proof}
  First note that the second estimate is
  an immediate consequence of \autoref{Prop_ProjectionBounds} and \autoref{Prop_LFComparison}, because
  \begin{equation*}
    \fq_\lambda a =
    \tau_\lambda(L_\lambda\mu_\lambda-\nu_\lambda
    F_{\fI}) \mu_\lambda a,
  \end{equation*}
  since $\tau_\lambda \nu_\lambda=0$.  Now, to estimate $\fp_\lambda$
  we define 
  \begin{equation*}
    \tilde\pi_\lambda\co \Omega^1(NQ,\fg_{E(\fI_\lambda)}) \to \Gamma(\fI^*V\fM)\subset \Omega^1(NQ,\fg_{E(\fI_\lambda)})
  \end{equation*}
  by
  \begin{equation*}
    (\tilde\pi_\lambda a)(x) \coloneq \sum_\kappa \int_{N_xQ} \<a,\kappa\> \kappa
  \end{equation*}
  and
  \begin{equation*}
    \tilde\sigma_\lambda\co \Omega^2(NQ,\fg_{E(\fI_\lambda)}) \to \Gamma(\Hom_\Phi(TQ,\fI^*V\fM))\subset \Omega^2(NQ,\fg_{E(\fI_\lambda)})
  \end{equation*}
  by
  \begin{equation*}
    (\tilde\sigma_\lambda \alpha)(x) \coloneq \sum_\beta \int_{N_xQ} \<\alpha,\beta\> \beta.
  \end{equation*}
  Here, at each point $x\in Q$, $\kappa$ runs through an orthonormal basis of $V\fM_{\fI(x)}$ and $\beta$ runs through an orthonormal basis of $\Hom_\Phi(T_xQ,V\fM_{\fI(x)})$.
  We set $\tilde\rho_\lambda\coloneq\id-\tilde\pi_\lambda$ and $\tilde\tau_\lambda\coloneq\id-\tilde\sigma_\lambda$.
  One can check that $\tilde\sigma_\lambda\tilde L_\lambda\tilde\rho_\lambda=0$.
  For $a$ supported in $V_{[0,\sigma)}$, which we can assume without loss of generality,
  \begin{align*}
    \fp_\lambda a
    &=\bar\sigma_\lambda (L_\lambda-\tilde L_\lambda)\rho_\lambda a
    + (\bar\sigma_\lambda-\tilde\sigma_\lambda) \tilde L_\lambda\rho_\lambda a
    + \tilde\sigma_\lambda \tilde L_\lambda(\rho_\lambda-\tilde\rho_\lambda) a \\
    &= \bar\sigma_\lambda\rI + \rII + \tilde\sigma_\lambda\rIII.
  \end{align*}
  The terms $\rII$ and $\rIII$ (resp.~$\rI)$ can be estimated similar to $\rI$ and $\rIII$ (resp.~$\rII$) in the proof of \autoref{Prop_LFComparison}.
\end{proof}

\subsection{Patching local inverses}

\begin{proof}[Proof of \autoref{Prop_RightInverse}]
  Fix $y\in \fY_\lambda$ and set
  \begin{equation*}
    u \coloneq \bar\sigma_\lambda y \qandq v \coloneq \tau_\lambda y.
  \end{equation*}

  \setcounter{step}{0}
  \begin{step}
    An approximate inverse for $u$.
  \end{step}

  Denote by $G_\fI$ a fixed right inverse of $F_\fI$ and set
  \begin{equation*}
    z \coloneq \mu_\lambda G_\fI \sigma_\lambda u.
  \end{equation*}
  We have
  \begin{equation*}
    \|z\|_{\fX_\lambda} \leq c  \|y\|_{\fY_\lambda}
  \end{equation*}
  and by \autoref{Cor_LFComparison} and \autoref{Prop_OffDiagonalEstimate} we have
  \begin{equation}\label{Eq_diff0}
    \|L_\lambda z - u\|_{\fY_\lambda}
    \leq c \lambda^{1-\alpha} \|y\|_{\fY_\lambda}.
  \end{equation}

  \begin{step}
    Choice of cut-off functions.
  \end{step}

  We construct an approximate inverse for $v$ by finding local
  approximate inverses and then patching these together.  This
  requires two kinds of cut-off functions.  The first kind is
  constructed as follows: Let $\chi\co[0,\infty)\to[0,1]$ denote the
  smooth-cut off function chosen in \autoref{Sec_Approx} which
  vanishes on $[0,1]$ and is equal to one on $[2,\infty)$.  We define
  $\chi_\lambda\co X\to[0,1]$ by
  \begin{equation*}
    \chi_\lambda(x)\coloneq\chi(r(x)/\sqrt\lambda).
  \end{equation*}
  Then
  \begin{equation*}
    \|\chi_\lambda\|_{C^{0,\alpha}_{0,0;\lambda}}\leq c.
  \end{equation*}
  Fix a small constant $\epsilon>0$, a large constant $N\gg 1$, and note that in the following we can choose the constant $c>0$ independent of $\epsilon$ and $N$.
  Throughout, we will make use of $\lambda\ll\epsilon$ and $\lambda \ll 1/N$.
  We can pick a finite number of points $\{x_\gamma : \gamma \in \Gamma \}\subset Q$ such that the balls $B_\epsilon(x_\gamma)$ cover all of $Q$ and a partition of unity $1 = \sum_{\gamma \in \Gamma} \chi_{\gamma}$ subordinate to this cover such that
  \begin{equation*}
    \|\chi_{\gamma}\|_{C^{0,\alpha}_{0,0;\lambda}\(\supp(1-\chi_\lambda)\)}
    \leq c\epsilon^{-\alpha}.
  \end{equation*}
  We can now write
  \begin{equation*}
    v = \sum_{\gamma \in \Gamma} v_{\gamma} + v_{0}
  \end{equation*}
  with
  \begin{equation*}
    v_{\gamma} \coloneq (1-\chi_\lambda) \chi_\gamma v \qandq
    v_{0} \coloneq \chi_\lambda v.
  \end{equation*}
  Although $v_0$ and the $v_\gamma$ depend on $\lambda$ we choose not to make this dependence explicit in order not to clutter the notation any more.
  By construction we have
  \begin{equation}\label{eq:v_est}
    \sum_\gamma \|v_{\gamma}\|_{C^{0,\alpha}_{-2,\delta;\lambda}}
    + \|v_0\|_{C^{0,\alpha}_{-2,\delta;\lambda}} 
    \leq c\epsilon^{-\alpha} \|v\|_{C^{0,\alpha}_{-2,\delta;\lambda}}.
  \end{equation}

  The second kind of cut-off functions is constructed as follows:
  We choose $\beta^\pm_{\lambda,N}\co X\to[0,1]$ such that
  \begin{align*}
    \beta^+_{\lambda,N}(x) &=
    \begin{cases}
      1 & r(x)\leq 2\sqrt\lambda \\
      0 & r(x)\geq 2N\sqrt\lambda
    \end{cases} \intertext{and}
    \beta^-_{\lambda,N}(x) &=
    \begin{cases}
      0 & r(x)\leq \sqrt\lambda/N \\
      1 & r(x)\geq \sqrt\lambda
    \end{cases}
  \end{align*}
  as well as
  \begin{equation}
    \label{Eq_beta}
    \|\rd \beta_{\lambda,N}^\pm\|_{C^{0,\alpha}_{-1,0;\lambda}} \leq
    c/\log(N)
    \qandq
    \|\beta_{\lambda,N}^\pm\|_{C^{0,\alpha}_{0,0;\lambda}} \leq
    c.
  \end{equation}
  This can be arranged by interpolating between $0$ and $1$ logarithmically, i.e., by defining $\beta^+_{\lambda,N}$ as an appropriate smoothing of $\log(2N\sqrt\lambda/r)/\log(N)$ in the intermediate region and similarly $\beta^-_{\lambda,N}$ as a smoothing of $\log(Nr/\sqrt\lambda)/\log(N)$.
  Moreover, we choose $\tilde\chi_\gamma\co Q\to[0,1]$ such that $\tilde\chi_\gamma$ equals one on $B_\epsilon(x_\gamma)$, $\tilde\chi_\gamma$ vanishes outside $B_{2\epsilon}(x_\gamma)$ and satisfies
  \begin{equation}\label{Eq_tildechi}
    \|\rd\tilde\chi_\gamma\|_{C^{0,\alpha}_{-1,0;\lambda}\(\supp\beta^+_{\lambda,N}\)}
    \leq cN\sqrt\lambda/\epsilon^{1+\alpha} \qandq 
    \|\tilde\chi_\gamma\|_{C^{0,\alpha}_{0,0;\lambda}\(\supp\beta^+_{\lambda,N}\)}
    \leq c.
  \end{equation}

  \begin{step}
    Construction of local approximate inverses.
  \end{step}

  Let $I_\gamma$ be the ASD instanton obtained by restricting $I=I(\fI)$ to $N_{x_\gamma}Q$.
  Using the identifications and the notation of \autoref{Sec_R8Model} we define
  \begin{equation*}
    \tilde w_\gamma \coloneq s_{1,\lambda}^{-1} \bL_{I_\gamma}^{-1} \rho_{I_\gamma}
    s_{2,\lambda} v_\gamma \qandq
    w_\gamma \coloneq \rho_\lambda \tilde\chi_\gamma\beta^+_{\lambda,N}\tilde w_\gamma.
  \end{equation*}
  where $\rho_{I_\gamma}\coloneq\id-\pi_{I_\gamma}$.  
  Under the identifications employed in \autoref{Sec_R8Model} the projections $\pi_\lambda$ and $\sigma_\lambda$ are identified.
  From $\sigma_\lambda v=0$ one can deduce that
  \begin{equation*}
    \|\pi_{I_\gamma} s_{2,\lambda}v_\gamma\|_{C^{0,\alpha}_{-2-\delta}} \leq
    c \epsilon 
    \|s_{2,\lambda}v_\gamma\|_{C^{0,\alpha}_{-2-\delta}}
    \leq c \epsilon
    \|v_\gamma\|_{C^{0,\alpha}_{-2,\delta;\lambda}}.
  \end{equation*}
  Using \autoref{Prop_LModelComparison} we conclude that
  \begin{equation}
    \label{Eq_wgamma}
    \|\tilde w_\gamma\|_{C^{1,\alpha}_{-1,\delta;\lambda}(V_{2\epsilon,\zeta})}
    \leq c\|s_{1,\lambda}\tilde w_\gamma\|_{C^{1,\alpha}_{-1+\delta}(U_{2\epsilon,\infty;\lambda})}
   \leq c\|v_\gamma\|_{C^{0,\alpha}_{-2,\delta;\lambda}}
  \end{equation}
  and
  \begin{equation}\label{Eq_diffgamma}
    \|L_\lambda \tilde w_\gamma - v_\gamma\|_{C^{0,\alpha}_{-2,\delta;\lambda}(V_{2\epsilon,\zeta})}
    \leq c\epsilon \|v_\gamma\|_{C^{0,\alpha}_{-2,\delta;\lambda}}
  \end{equation}
  Since $\pi_{I_\gamma}(s_{1,\lambda}\tilde w_\gamma)=0$, it follows that
  \begin{equation*}
    \|\tilde\pi_{I_\gamma} s_{1,\lambda} \tilde w_\gamma\|_{C^{1,\alpha}_{-1+\delta}(U_{2\epsilon,\infty;\lambda})}
    \leq c \epsilon \|s_{1,\lambda}\tilde w_\gamma\|_{C^{1,\alpha}_{-1+\delta}(U_{2\epsilon,\infty;\lambda})}
  \end{equation*}
  here $\tilde\pi_{I_\gamma}$ is defined like $\pi_{I_\gamma}$ but with $\ker \delta_{I|_{N_{\exp_{x_\gamma}(\lambda\cdot-)}Q}}$ instead of $\ker \delta_{I|_{N_{x_\gamma}Q}}$.
  Therefore,
  \begin{equation}\label{Eq_piwgamma}
    \|\bar\pi_\lambda
    w_\gamma\|_{C^{1,\alpha}_{-1,\delta;\lambda}}
    \leq c \epsilon \|v_\gamma\|_{C^{1,\alpha}_{-1,\delta;\lambda}}
  \end{equation}
  and it follows that
  \begin{align*}
    \sum_\gamma \|w_\gamma\|_{C^{1,\alpha}_{-1,\delta;\lambda}}
    &\leq c(1+N\sqrt\lambda/\epsilon^{1+\alpha}+1/\log(N)) \sum_\gamma \|v_\gamma\|_{C^{0,\alpha}_{-2,\delta;\lambda}} \\
    &\leq c\epsilon^{-\alpha}(1+N\sqrt\lambda/\epsilon^{1+\alpha}+1/\log(N)) \|v\|_{C^{1,\alpha}_{-1,\delta;\lambda}}.
  \end{align*}

  By \autoref{Prop_XminusQModelComparison}, $w_0 \coloneq \beta^-_{\lambda,N}R_{A_0} v_0$, with $R_{A_0}$ as in \autoref{Prop_XminusQInverse}, satisfies
  \begin{equation}\label{Eq_w0_est}
    \|w_0\|_{\fX_\lambda}
    \leq c \|v_0\|_{\fY_\lambda}.
  \end{equation}

  Combining all of the above we see that the $\tilde R_\lambda\co \fY_\lambda\to \fX_\lambda$ defined by
  \begin{equation*}
    \tilde R_\lambda y
    \coloneq z + \sum_{\gamma}  w_\gamma
    +  w_0.
  \end{equation*}
  is bounded by $c\epsilon^{-\alpha}(1+N\sqrt\lambda/\epsilon^{1+\alpha}+1/\log(N))$.

  \begin{step}
    $\tilde R_\lambda$ is an approximate right inverse to $L_\lambda$.
  \end{step}

  We need to estimate the three types of terms
  \begin{align*}
    \rI &\coloneq \|L_\lambda z - u\|_{\fY_\lambda}, \\
    \rII_\gamma &\coloneq  \|L_\lambda w_\gamma - v_\gamma\|_{\fY_\lambda} \qand \\
    \rIII &\coloneq \|L_\lambda w_0 - v_0\|_{\fY_\lambda}.
  \end{align*}
  We have already treated $\rI$ with \eqref{Eq_diff0}.
  Now,
  \begin{align*}
    \rII_\gamma 
    = \|L_\lambda w_\gamma - v_\gamma\|_{\fY_\lambda}
    &\leq \lambda^{-\delta/2}\|L_\lambda \rho_\lambda \tilde\chi_\gamma \beta^+_{\lambda,N} \tilde w_\gamma - v_\gamma\|_{C^{0,\alpha}_{-2,\delta;\lambda}} \\
    &\qquad+
    \lambda\|\sigma_\lambda L \rho_\lambda \tilde\chi_\gamma \beta^+_{\lambda,N} \tilde w_\gamma - \sigma_\lambda v_\gamma\|_{C^{0,\alpha}}
  \end{align*}
  Using \eqref{Eq_wgamma}, \autoref{Prop_OffDiagonalEstimate} and the fact that $\pi_\lambda v = 0$ the last term can be seen to be bounded by $c\lambda^{1-\alpha}\|v_\gamma\|_{\fY_\lambda}$.
  To control the first term use the fact that on the support of $v_\gamma$ we have $\tilde\chi_\gamma\beta^+_{\lambda,N}=1$, \eqref{Eq_beta}, \eqref{Eq_tildechi}, \eqref{Eq_wgamma}, \eqref{Eq_diffgamma} and \eqref{Eq_piwgamma} to derive
  \begin{align*}
    \|L_\lambda \rho_\lambda \tilde\chi_\gamma\beta^+_{\lambda,N}\tilde w_\gamma - v_\gamma\|_{C^{0,\alpha}_{-2,\delta;\lambda}} 
    &\leq c\|L_\lambda \tilde w_\gamma - v_\gamma\|_{C^{0,\alpha}_{-2,\delta;\lambda}(V_{2\epsilon,\zeta})} \\
    &\qquad
    + c\|\rd(\tilde\chi_\lambda \beta^+_{\lambda,N})\|_{C^{0,\alpha}_{-1,0;\lambda}(\supp\beta^+_{\lambda,N})}
    \|\tilde w_\gamma\|_{C^{0,\alpha}_{-1,\delta;\lambda}(V_{2\epsilon,\zeta})} \\
    &\qquad
    + c\|\bar\pi_\lambda \tilde\chi_\gamma\beta^+_{\lambda,N}\tilde w_\gamma\|_{C^{1,\alpha}_{-1,\delta;\lambda}} \\
    &\leq c(\epsilon + 1/\log(N) + N\sqrt\lambda/\epsilon) \|v_\lambda\|_{C^{0,\alpha}_{-2,\delta;\lambda}}.
  \end{align*}
  Similarly,
  \begin{equation*}
    \rIII \leq c(\sqrt\lambda+1/\log(N)) \|y\|_{\fY_\lambda}
  \end{equation*}

  Putting everything together we obtain
  \begin{equation*}
    \|L_\lambda\tilde R_\lambda y -y\|_{\fY_\lambda}
    \leq c\epsilon^{-\alpha}(\epsilon+1/\log N+N\sqrt\lambda/\epsilon) \|y\|_{\fY_\lambda}.
  \end{equation*}
  By choosing $\epsilon$ small enough, $N$ large enough and $\lambda$ small enough we can make the factor in front of $\|y\|_{\fY_\lambda}$ arbitrarily small.

  \begin{step}
    Construction of $R_\lambda$.
  \end{step}

  We can arrange that
  \begin{equation*}
    \|L_\lambda \tilde R_\lambda y - y\|_{\fY_\lambda}
    \leq \frac12 \|y\|_{\fY_\lambda}.
  \end{equation*}
  for all $\lambda\in(0,\Lambda]$;
  hence, the series
  \begin{equation*}
    R_\lambda
      \coloneq \tilde R_\lambda (L_\lambda \tilde R_\lambda)^{-1}
      = \tilde R_\lambda \sum_{k=0}^\infty \(\id -L_\lambda
      \tilde R_\lambda\)^k
  \end{equation*}
  converges and constitutes a right inverse for $L_\lambda$.
  Clearly, $R_\lambda$ is bounded uniformly with respect to $\lambda\in(0,\Lambda]$.
\end{proof}


%% file: pf.tex
\section{Conclusion of  the proof of \autoref{Thm_A}}

The last ingredient we need for the proof of \autoref{Thm_A} is the following estimate on the polarisation
\begin{equation*}
  Q(a_1,a_2)\coloneq\frac12\pi_7([a_1\wedge a_2])
\end{equation*}
of the quadratic form $Q$ appearing in \eqref{Eq_A}.

\begin{prop}
  \label{Prop_QuadraticEstimate}
  There is a constant $c>0$ such that for all $\lambda \in (0,\Lambda]$ we have
  \begin{align*}
    &\|\tau_\lambda Q(a_1,a_2)\|_{C^{0,\alpha}_{-2,\delta;\lambda}} \\
    &\qquad\leq c\lambda^{-\alpha}\Bigl(
    \|\rho_\lambda a_1\|_{C^{0,\alpha}_{-1,\delta;\lambda}}
    \cdot\|\rho_\lambda a_2\|_{C^{0,\alpha}_{-1,\delta:\lambda}}
    +\|\rho_\lambda a_1\|_{C^{0,\alpha}_{-1,\delta;\lambda}}
    \cdot\|\pi_\lambda a_2\|_{C^{0,\alpha}} \\
    &\qquad\qquad\qquad+\|\pi_\lambda a_1\|_{C^{0,\alpha}}
    \cdot\|\rho_\lambda a_2\|_{C^{0,\alpha}_{-1,\delta;\lambda}}
    +\|\pi_{\lambda}a_1\|_{C^{0,\alpha}}\|\pi_{\lambda}a_2\|_{C^{0,\alpha}}\Bigr)
  \end{align*}
  and
  \begin{align*}
    &\lambda\|\sigma_\lambda Q(a_1,a_2)\|_{C^{0,\alpha}} \\
    &\qquad\leq c\lambda^{-\alpha}\Bigl(
    \|\rho_\lambda a_1\|_{C^{0,\alpha}_{-1,\delta;\lambda}}
    \cdot\|\rho_\lambda a_2\|_{C^{0,\alpha}_{-1,\delta;\lambda}}
    +\|\rho_\lambda a_1\|_{C^{0,\alpha}_{-1,\delta;\lambda}}
    \cdot\|\pi_\lambda a_2\|_{C^{0,\alpha}} \\
    &\qquad\qquad\qquad+\|\pi_\lambda a_1\|_{C^{0,\alpha}}
    \cdot\|\rho_\lambda a_2\|_{C^{0,\alpha}_{-1,\delta;\lambda}}
    +\lambda \|\pi_\lambda a_1\|_{C^{0,\alpha}}
    \cdot \|\pi_\lambda a_2\|_{C^{0,\alpha}}\Bigr).
  \end{align*}
  In particular,
  \begin{equation*}
    \|Q(a_1,a_2)\|_{\fY_\lambda} \leq c \lambda^{-2-\delta/2} \|a_1\|_{\fX_\lambda}\|a_2\|_{\fX_\lambda}
  \end{equation*}
\end{prop}

\begin{proof}
  The first estimate is an immediate consequence of \autoref{Prop_Multiplication} and \autoref{Prop_ProjectionBounds}.
  For the second estimate we only have to explain why we get a factor $\lambda$ in front of $\|\pi_\lambda a_1\|_{C^{0,\alpha}} \cdot \|\pi_\lambda a_2\|_{C^{0,\alpha}}$. 
  Note that
  \begin{equation*}
   \tilde\sigma_\lambda \pi_7^0\(\mu_\lambda\hat\fI_1 \wedge \mu_\lambda\hat\fI_2\) = 0
  \end{equation*}
  because of \autoref{Prop_Lambda2SplittingHyperkahlerPair} (the $\Lambda^2_4$--component already vanishes).
  Arguing as in the proof of \autoref{Prop_LFComparison} we see that we gain a factor of $\lambda$.
\end{proof}

Setting $\tilde Q_\lambda=Q\circ R_\lambda$, \autoref{Eq_A} becomes
\begin{equation*}
  x + \tilde Q_\lambda(x) + \pi_7(F_{A_\lambda}) = 0.
\end{equation*}
In view of \autoref{Prop_ErrorEstimate}, \autoref{Prop_RightInverse} and \autoref{Prop_QuadraticEstimate}, this equation can be solved by appealing to the following consequence of Banach's fixed-point theorem.

\begin{lemma}[{\cite[Lemma~7.2.23]{Donaldson1990}}]
  \label{Lemma_BFPS}
  Let $X$ be a Banach space and let $T \co X\to X$ be a smooth map with $T(0)=0$.
  Suppose there is a constant $c>0$ such that
  \begin{equation*}
    \|Tx-Ty\|\leq c\(\|x\|+\|y\|\)\|x-y\|.
  \end{equation*}
  Then if $y\in X$ satisfies $\|y\|\leq \frac{1}{10c}$, there exists a unique $x\in X$ with $\|x\|\leq \frac{1}{5c}$ solving
  \begin{equation*}
    x+Tx=y.
  \end{equation*}
  The unique solution satisfies $\|x\|\leq 2\|y\|$.
\end{lemma}

Elliptic regularity implies that $A_\lambda+a$ is smooth.
Since $a$ is small, the existence of a right inverse of $L_{A_\lambda}$ guarantees the existence of a right inverse of $L_{A_\lambda+a}$;
hence, $A_\lambda+a$ is irreducible and unobstructed.
\qed


%% file: fib.tex
\section{Proof of \autoref{Thm_B}}
\label{Sec_CayleyFibration}

Since $\Hol(g_\Phi)=\Spin(7)$, $b^1=b^2_7=0$ \cite[Proposition 10.6.5]{Joyce2000} and thus the product connection $\theta$ on the trivial $\SU(2)$--bundle is unobstructed.
It is reducible; however, does not cause any problems, see \autoref{Rmk_Reducible}.
We have $\ind L_\theta=-3$.
If we choose $M$ as in \autoref{Ex_ChargeOne}, then
\begin{equation*}
  \fM
  =
  \(\Re(\Hom(\C^2,\slS^+))\setminus\{0\}\)/\Z_2
  \times \Re(\slS^+\otimes U)
\end{equation*}
By \autoref{Ex_ChargeOneFueter}, the Fueter operator lifts to the Dirac operator
\begin{equation*}
  \slD\co
  \Gamma(\Re(\Hom(\C^2,\slS^+)\oplus \slS^+\otimes U))
  \to
  \Gamma(\Re(\Hom(\C^2,\slS^-)\oplus \slS^-\otimes U)).
\end{equation*}
Arguing as in \autoref{Prop_CayleyFibration}, we see that $\slD$ is surjective and has an $8$--dimensional kernel and all non-zero elements of the kernel are no-where vanishing, provided the metric on $Q$ is sufficiently close to a hyperkähler metric and the induced connection on $NQ$ is almost flat.
We can thus apply \autoref{Thm_A} and obtain a $5$--dimensional family of $\Spin(7)$--instantons over $X$.
A similar argument also proves the last assertion of the theorem.
\qed


%% file: index.tex
\section{Comparing index formulae}

\begin{prop}
\label{Prop_IndexDifference}
  Let $(X,\Phi)$ be a compact $\Spin(7)$--manifold, let $Q$ be a Cayley submanifold of $X$ and let $E_0$ and $E$ be $\SU(2)$--bundles over $X$ which are related by 
  \begin{equation*}
    c_2(E) = c_2(E_0)+\PD[Q].
  \end{equation*}
  If $A$ is a connection on $E$ and $A_0$ is a connection on $E_0$, then
  \begin{align}
    \label{Eq_IndexDifference}
    \ind L_{A} &=\ind L_{A_0} + \ind F_Q + \ind \mathring F + \frac53 \int_Q e(\Re(\Hom(E_0,\slS^+_Q)))
  \end{align}
  where $L_A$ and $L_{A_0}$ are as in \eqref{Eq_Spin7InstantonLinearisation} and
  \begin{equation*}
    \mathring F =
    \slD\co
    \Gamma(\Re(\Hom(E_\infty,\slS^+)))
    \to
    \Gamma(\Re(\Hom(E_\infty,\slS^-)))
  \end{equation*}
  as in \autoref{Ex_ChargeOneLinearisedFueter}.
\end{prop}

In the situation of \autoref{Prop_IndexDifference} whenever \autoref{Thm_A} can be applied $e(\Re(\slS^+_Q\otimes E_0))$ vanishes.
This is because in those situation the Fueter section $\fI$ gives rise to a no-where vanishing section of $\Re(\slS^+_Q\otimes E_0)$.  
Hence, \eqref{Eq_IndexDifference} can be taken as evidence that \autoref{Thm_A} gives a description of an open subset of the moduli space of $\Spin(7)$--instantons.
(Note that the gluing parameter $\lambda$ is already contained in $\ind \mathring F_{\fI}$.)

\begin{proof}[Proof of \autoref{Prop_IndexDifference}]
  By \autoref{Prop_Spin7InstantonIndex} and \eqref{Eq_p1nq} we have
  \begin{align*}
    \ind L_{A_\lambda} -  \ind L_{A_0}
    &= - \frac{1}{6} \int_Q p_1(X) - \frac43 [Q]\cdot[Q]
    - \frac83 \int_Q c_2(E_0) \\
    &= - \sigma(Q) - \frac13 \chi(Q) - [Q]\cdot[Q] - \frac83 \int_Q c_2(E_0).
  \end{align*}
  By \eqref{Eq_CayleyIndex} and \eqref{Eq_FueterIndex} we have
  \begin{equation*}
    \ind F_Q + \ind \mathring F_{\fI} = \frac14 \sigma(Q) + \frac12\chi(Q) -[Q]\cdot[Q] - \int_Q c_2(E_0).
  \end{equation*}
  Using \eqref{Eq_ep1c2} and \eqref{Eq_C2SpinorBundle}
  \begin{align*}
    \int_Q e(\Re(\Hom(E_0,\slS^+_Q))
    &=
      \int_Q e(\Re(E_0^*\otimes \slS^+_Q)) \\
    &=
      \int_Q e(\Re(E_0\otimes \slS^+_Q)) \\
    &=
      -\int_Q c_2(E_0) - \frac34 \sigma(Q) - \frac12 \chi(Q).
  \end{align*}
  Verifying \eqref{Eq_IndexDifference} is now straight-forward.
\end{proof}
